\documentclass[a4paper]{amsart}
\usepackage[all]{xy}
\usepackage{subfig}
\usepackage{mathrsfs}
\usepackage{amsmath}
\usepackage{amssymb,latexsym}
\usepackage{mathrsfs}
\usepackage{graphics}
\usepackage{latexsym}
\usepackage{psfrag}
\usepackage{verbatim}
\usepackage{graphicx}
\usepackage[usenames]{color}
\usepackage{pifont,marvosym}

\theoremstyle{plain}

\newtheorem{theorem}{Theorem}[section]

\newtheorem{lemma}[theorem]{Lemma}

\newtheorem{corollary}[theorem]{Corollary}

\newtheorem{proposition}[theorem]{Proposition}

\theoremstyle{definition}

\newtheorem{definition}[theorem]{Definition}

\theoremstyle{remark}
\newtheorem{remark}[theorem]{Remark}

\DeclareSymbolFont{AMSb}{U}{msb}{m}{n}
\DeclareMathSymbol{\N}{\mathalpha}{AMSb}{"4E}
\DeclareMathSymbol{\R}{\mathalpha}{AMSb}{"52}
\DeclareMathSymbol{\Z}{\mathalpha}{AMSb}{"5A}
\DeclareMathSymbol{\D}{\mathalpha}{AMSb}{"44}
\DeclareMathSymbol{\s}{\mathalpha}{AMSb}{"53}

\newcommand{\Q}{\ensuremath{\mathbb{Q}}}

\newcommand{\sX}{\scriptscriptstyle{X}}

\newcommand{\cH}{\mathcal{H}}
\newcommand{\cJ}{\mathcal{J}}
\newcommand{\cU}{\mathcal{U}}
\DeclareMathOperator{\Cpl}{Cpl}

\DeclareMathOperator{\tr}{tr}

\DeclareMathOperator{\supp}{supp}

\DeclareMathOperator{\sfd}{\mathsf{d}}

\DeclareMathOperator{\ric}{Ric}
\DeclareMathOperator{\sect}{Sec}

\DeclareMathOperator{\Geo}{Geo}

\DeclareMathOperator{\Ent}{Ent}

\newcommand{\ee}{{\rm e}}
\begin{document}

\title[Sectional and intermediate Ricci  bounds via Optimal Transport] {Sectional and intermediate Ricci curvature lower bounds via Optimal Transport} 
\author{Christian Ketterer}\thanks{C. Ketterer: Freiburg University, Mathematics Department, email: ckettere@math.toronto.edu}
\author {Andrea Mondino} \thanks{A. Mondino: University of Warwick,  Mathematics Institute, email: A.Mondino@warwick.ac.uk}


\bibliographystyle{plain}

\begin{abstract}
The goal of the paper is to give an optimal transport characterization of sectional curvature lower (and upper)  bounds for smooth $n$-dimensional Riemannian manifolds. More generally we characterize, via optimal transport, 
lower bounds on the so called $p$-Ricci curvature which  corresponds to taking the trace of the Riemann curvature tensor on $p$-dimensional planes, $1\leq p\leq n$.  Such characterization roughly consists on a convexity condition of the $p$-Renyi entropy along $L^{2}$-Wasserstein geodesics,  where the role of   reference measure is played by the $p$-dimensional Hausdorff measure.  As application we establish a new Brunn-Minkowski type inequality involving  $p$-dimensional submanifolds and the $p$-dimensional Hausdorff measure.
  \end{abstract}

\maketitle

\tableofcontents
\section{Introduction}
The interplay between Ricci curvature and optimal transport is well known and it has been a topic of tremendous interest in the last years.   On the other hand 
it seems to be still an open problem to find the link between \emph{sectional} curvature bounds (and more generally intermediate Ricci curvature bounds) and optimal transportation. The goal of the paper is to address such a question.
\\

Inspired by the pioneering work on  Ricci curvature lower bounds  via optimal transport by Sturm and von Renesse  \cite{SVR}, later extended to non-smooth spaces in the foundational  works of Lott-Villani \cite{lottvillani:metric} and  Sturm \cite{stugeo1, stugeo2}, we analyze convexity properties  of the $p$-Renyi entropy along $L^{2}$-Wasserstein geodesics,  where the role of  the  reference measure is played  here by the $p$-dimensional Hausdorff measure. 
In a first approximation, one can think of studying the convexity of the $p$-Renyi entropy along an $L^{2}$-Wasserstein geodesics made of probability measures concentrated on $p$-dimensional submanifolds of $M$. 
\\The study of optimal transportation between measures supported on arbitrary submanifolds in an arbitrary Riemannian manifold seems to be  quite a new topic in the literature. Nevertheless  several authors 
treated remarkable particular cases and related questions: 
\begin{itemize}
\item Gangbo-McCann \cite{GM1} proved results about optimal
transport between measures supported on hyper-surfaces in Euclidean space;
\item McCann-Sosio \cite{MS} and Kitagawa-Warren \cite{KW} gave more refined  results about optimal transport between two
measures supported on a  codimension one sphere in Euclidean space;
\item Castillon \cite{Cast} considered optimal transport between
a measure supported on a submanifold of Euclidean space and a measure supported on a
linear subspace;
\item Lott \cite{lott} characterized the tangent cone (in the $W_{2}$-metric) to a probability measure supported on a smooth submanifold of a Riemannian manifold.
\end{itemize}

In order to state the results, let us introduce some notation (for more details see  Section \ref{sec:Prel}). Let $(M^{n},g)$ be a smooth, complete, $n$-dimensional Riemannian manifold without boundary. For $p=\{1,\ldots,n\}$, denote by $\cH^{p}$ the $p$-dimensional Hausdorff measure and consider the space  $\mathcal{P}_{c}(M,\mathcal{H}^p)$ of  probability measures with compact support which are absolutely continuous with respect to $\cH^{p}$.  Given $1\leq p\leq p'<\infty$, the $p'$-Renyi entropy with respect to $\mathcal{H}^p$ is defined as
\begin{align*}
S_{p'}(\cdot|\mathcal{H}^p):\mathcal{P}_c(M,\mathcal{H}^p) \rightarrow [-\infty,0],\ \ S_{p'}(\mu|\mathcal{H}^p)=-\int\rho^{1-\frac{1}{p'}}d\mathcal{H}^p,
\end{align*}
where $\rho$ is the density of $\mu$ with respect to $\mathcal{H}^p$, i.e. $\mu=\rho \mathcal{H}^p$.
Note that in the borderline case $p'=p=1$, one gets 
$$S_{1}(\mu|\mathcal{H}^{1})=- \mathcal{H}^{1}(\supp(\mu)).$$

The (relative) Shannon entropy is defined by
\begin{align*}
\Ent(\cdot|\mathcal{H}^p):\mathcal{P}_c(M,\mathcal{H}^p) \rightarrow [-\infty,\infty],\ \ \Ent(\mu|\mathcal{H}^p)= \lim_{\varepsilon\downarrow 0} \int_{\{\rho>\varepsilon\}} \rho  \log \rho \, d\mathcal{H}^p.
\end{align*}
This coincides with  $ \int_{\{\rho>0\}} \rho  \log \rho \, d\mathcal{H}^p$, provided that $ \int_{\{\rho\geq 1\}} \rho  \log \rho \, d\mathcal{H}^p < \infty$, and $\Ent(\mu|\mathcal{H}^p):=\infty$ otherwise. 
Recall also the definition of the distortion coefficients. Given  $K\in \R$ , we set for $(t,\theta) \in[0,1] \times \R^{+}$, 
\begin{equation*}
\sigma_{K,1}^{(t)}(\theta):= 
\begin{cases}
\infty & \textrm{if}\ K\theta^{2} \geq \pi^{2}, \crcr
\displaystyle  \frac{\sin(t\theta\sqrt{K})}{\sin(\theta\sqrt{K})} & \textrm{if}\ 0< K\theta^{2} <  \pi^{2}, \crcr
t & \textrm{if} \ K \theta^{2}=0,  \crcr
\displaystyle   \frac{\sinh(t\theta\sqrt{-K})}{\sinh(\theta\sqrt{-K})} & \textrm{if}\ K\theta^{2} \leq 0.
\end{cases}
\end{equation*}
\\A subset $\Sigma\subset M$ is said  to be \emph{countably $\cH^{p}$-rectifiable} if, up to a $\cH^{p}$-negligible subset, it can be covered by countably many $p$-dimensional Lipschitz submanifolds. We say that a  $W_2$-geodesic $\{\mu_t\}_{t\in [0,1]}$ is \emph{countably $\mathcal{H}^p$-rectifiable} if  for every $t\in [0,1]$ the measure $\mu_{t}\in  \mathcal{P}_c(M,\mathcal{H}^p)$ is concentrated on a countably $\mathcal{H}^p$-rectifiable set $\Sigma_{t}\subset M$ (see Section \ref{sec:W2Rect} for a  through discussion of rectifiable $W_{2}$-geodesics and in particular Remark \ref{rem:rectmut} for a sufficient  generic condition of rectifiability). 

Our first main result is an optimal transport characterization of sectional curvature upper bounds.

\begin{theorem}[OT characterization of sectional curvature upper bounds, Theorem \ref{thm:USB}]\label{thm:UB}
Let $(M,g)$ be a complete Riemannian manifold without boundary and let $K\geq 0$. Then the following statements {\rm (i)} and {\rm (ii)} are equivalent: 
\begin{itemize}
\item[(i)] The sectional curvature of $(M,g)$ is bounded above by $K$.
\medskip
\item[(ii)]
Let $\{\mu_t\}_{t\in [0,1]}$ be a countably $\mathcal{H}^1$-rectifiable $W_2$-geodesic, and let $\Pi$ be the corresponding dynamical optimal plan. Then, if $t_0,t_1\in (0,1)$ and  $\tau(s)=(1-s)t_0+st_1$, it holds
\begin{align*}
\cH^{1}(\supp \mu_{\tau(s)})\leq \int\left[\sigma_{K,1}^{(1-s)}(|\dot{\gamma\circ\tau}|)\rho_{t_0}(\gamma(t_0))^{-1}+\sigma_{K,1}^{(s)}(|\dot{\gamma\circ \tau}|)\rho_{t_1}(\gamma(t_1))^{-1}\right]d\Pi(\gamma), \ \ \forall s \in [0,1],
\end{align*}
where $\rho_t$ is the density of $\mu_t$ with respect to $\mathcal{H}^1$.
\end{itemize}
In the case of $K=0$ the inequality in {\rm (ii)} becomes
\begin{align*}
\cH^{1}(\supp \mu_{\tau(s)})\leq (1-s)\cH^{1}(\supp\mu_{t_0}) + s\cH^{1}(\supp\mu_{t_1}), \ \ \forall s \in [0,1].
\end{align*}
\end{theorem}
See Remark \ref{rem:UBK>0} for the motivation why the upper bound $K$ must be non-negative. Let us also stress that  in the assertion (ii)  one cannot relax the assumption to  $t_0,t_1\in [0,1]$, see Remark \ref{rem:SharpUB} for a  counterexample.
\\

The second main result is an optimal transport characterization of sectional curvature lower bounds.  In order to state it, some more notation must be introduced.
First of all, given a  $\mathcal{H}^p$-rectifiable $W_2$-geodesic  $\{\mu_t\}_{t\in [0,1]}$, thanks to the Monge-Mather shortening principle \cite[Theorem 8.5]{viltot}  we know that, for every $t \in [0,1]$, $\mu_{t}=(T^{t}_{1/2})_{\sharp} \mu_{1/2}$ with $T^{t}_{1/2}:\Sigma_{1/2}\to \Sigma_{t}$ Lipschitz. 
For $\mu_{1/2}$-a.e. $x$ we can set (see  Lemma \ref{lem:vDiff} and Remark \ref{rem:defB} for the details) 
\begin{equation*}\label{eq:defBxt}
{B}_{x}(t): T_{x} \Sigma_{1/2} \to  T_{\gamma_{x}(t)} \Sigma_{t}, \quad {B}_{x}(t):=DT_{1/2}^{t}(x)
\quad \forall t\in [0,1].
\end{equation*} 
In  Lemma \ref{B} we will prove  a Monge-Amp\`ere inequality implying that  $B_x(t)$ is invertible.   Let $\gamma_{x}(t):=T^{t}_{1/2}(x)$ be a geodesic performing the transport  and consider 
\begin{align*}
\mathcal{U}_{x}(t)&:= (\nabla_tB_{x}(t)){B_{x}}(t)^{-1}: T_{\gamma_{x}(t)}\Sigma_t\rightarrow T_{\gamma_{x}(t)}M, \\
\mathcal{U}_{x}^{\perp}(t)&:=[\mathcal{U}_{x}(t)]^{\perp} :  T_{\gamma_{x}(t)}\Sigma_t\rightarrow (T_{\gamma_{x}(t)}\Sigma_t)^{\perp},
\end{align*}
where $\nabla_{t}$ denotes the covariant derivative along $\gamma_{x}(t)$ in $M$  and $\perp$ is the orthogonal projection on the orthogonal complement $(T_{\gamma(t)}\Sigma_t)^{\perp}$ of $T_{\gamma(t)}\Sigma_t$. If  $|\dot{\gamma}_{x}|\neq 0$,  we set
 \begin{equation*}
 \kappa_{\gamma_{x}}:[0,|\dot{\gamma}_{x}|]\rightarrow \mathbb{R}, \quad  \kappa_{\gamma_{x}}(|\dot{\gamma_{x}}|\,t) \, | \dot{\gamma}_{x}|^2:= \left\|\mathcal{U}_{x}^{\perp}(t) \right\|^{2}, \quad \forall t \in [0,1],
 \end{equation*}
 if $|\dot{\gamma}_{x}|=0$,  we set   $\kappa_{\gamma_{x}}(0)=0$.
We now introduce the generalized distortion coefficients $\sigma_{\kappa}$ associated to a continuous function $\kappa:[0,\theta] \to \mathbb{R}$ (cf. \cite{ketterer4}). First of all, the generalized $\sin$-function associated to $\kappa$, denoted by $\sin_{\kappa}$,  is defined as the unique solution $v:[0,\theta] \to \R$ of the equation
\begin{align*}
v''+\kappa v=0 \ \ \& \ \ v(0)=0, \ v'(0)=1.
\end{align*}
The generalized distortion coefficients $\sigma_{\kappa}^{(t)}(\theta)$, for $t \in [0,1]$ and $\theta>0$, are defined as 

\begin{align*}
\sigma_{\kappa}^{(t)}(\theta):=\begin{cases}
                              \frac{\sin_{\kappa}(t\theta)}{\sin_{\kappa}(\theta)} \ \ & \ \ \text{if} \ \sin_{\kappa}(t \theta)>0 \; {  \text{ for all } t \in [0,1], }\\
                            {  \infty } \ \ & \ \ \mbox{otherwise}.
                              \end{cases}
\end{align*}
In the case $\kappa=K={\rm const}$ one has $\sigma_{\kappa}^{(t)}(\theta)=\sigma_{K,1}^{(t)}(\theta)$.
It is convenient to also set $\sigma_{\kappa}^{(\cdot)}(0)\equiv 1$,
$\kappa^-(t)=\kappa(\theta-t)$ and $\kappa^{+}(t):=\kappa(t)$.
Finally, consider the  Green function ${\rm g}:[0,1]\times [0,1]\to [0,1]$ given by
\begin{equation*}
  {\rm g}(s,t):=
  \begin{cases}
    (1-s)t&\text{if }t\in[0,s],\\
    s(1-t)&\text{if }t\in [s,1].
  \end{cases}
\end{equation*}

We can now state the  optimal transport characterization of sectional curvature lower bounds.

\begin{theorem}[OT characterization of sectional curvature lower bounds]\label{Cor:mainSec}
Let $(M,g)$ be a complete $n$-dimensional Riemannian manifold without boundary and fix $K \in \R$.
\begin{itemize}
\item If $K\geq 0$ the next  conditions are equivalent:
\begin{itemize}\item[(i)] $M$ has  sectional curvature bounded from below by $K$.
 \item[(ii)] Let $p \in \{2,\ldots, n\}$  be arbitrary, let $\{\mu_{t}\}_{t \in [0,1]}$ be a $\cH^{p}$-rectifiable $W_{2}$-geodesic  and  $\Pi$ be the corresponding dynamical optimal plan. Then, for any $p'\geq p$, for all $t \in [0,1]$ it holds
  \begin{align*}
 S_{p'}(\mu_t|\mathcal{H}^p)\leq - \int \left[\sigma_{((p-1)K-\kappa_{\gamma}^{-})/p'}^{(1-t)}(|\dot{\gamma}|)\, \rho_0^{-\frac{1}{p'}}(\gamma(0)) + \sigma_{((p-1)K-\kappa_{\gamma}^+)/p'}^{(t)}(|\dot{\gamma}|)   \, \rho_1^{-\frac{1}{p'}}(\gamma(1))\right]d\Pi(\gamma).
\end{align*}
\item[(ii)']  The condition {\rm (ii)} holds for $p=2$.
\item[(iii)] Let $p \in \{2,\ldots, n\}$ be arbitrary, $\{\mu_t\}_{t\in [0,1]}$ and  $\Pi$ be as in {\rm (ii)}. Then for all $t \in [0,1]$ it holds
\begin{align*}
\Ent(\mu_t|\cH^{p})\leq (1-t)\Ent (\mu_0|\cH^{p})+t\Ent(\mu_1|\cH^{p}) - \int \int_0^1 {\rm g}(s,t)\, |\dot{\gamma}|^2 \,  ((p-1)K-\kappa_{\gamma}(s|\dot{\gamma}|) )\,  ds \, d\Pi(\gamma).
\end{align*}
\item[(iii)']  The condition {\rm (iii)} holds for $p=2$.
\end{itemize}

\item If $K\leq 0$ the next conditions are equivalent:
\begin{itemize}
\item[(i)] $M$ has  sectional curvature bounded from below by $K$.
 \item[(ii)] Let $p \in \{1,\ldots, n\}$ be arbitrary, let $\{\mu_{t}\}_{t \in [0,1]}$ be a $\cH^{p}$-rectifiable $W_{2}$-geodesic  and  $\Pi$ be the corresponding dynamical optimal plan. Then, for any $p'\geq p$, for all $t \in [0,1]$ it holds
  \begin{align*}
 S_{p'}(\mu_t|\mathcal{H}^p)\leq - \int \left[\sigma_{(\bar{K}-\kappa_{\gamma}^{-})/p'}^{(1-t)}(|\dot{\gamma}|)\, \rho_0^{-\frac{1}{p'}}(\gamma(0)) + \sigma_{(\bar{K}-\kappa_{\gamma}^+)/p'}^{(t)}(|\dot{\gamma}|)   \, \rho_1^{-\frac{1}{p'}}(\gamma(1))\right]d\Pi(\gamma),
\end{align*}
 where $\bar{K}:=\min\{p, n-1\} K$. 
 \item[(ii)']  The condition {\rm (ii)} holds for $p=1$.
\item[(iii)] Let $p \in \{1,\ldots, n\}$  be arbitrary, $\bar{K}$, $\{\mu_t\}_{t\in [0,1]}$ and  $\Pi$ be as in {\rm (ii)}. Then for all $t \in [0,1]$ it holds
\begin{align*}
\Ent(\mu_t|\cH^{p})\leq (1-t)\Ent (\mu_0|\cH^{p})+t\Ent(\mu_1|\cH^{p}) - \int \int_0^1 {\rm g}(s,t)\, |\dot{\gamma}|^2 \,  (\bar{K}-\kappa_{\gamma}(s|\dot{\gamma}|) )\,  ds \, d\Pi(\gamma).
\end{align*}
\item[(iii)']  The condition {\rm (iii)} holds for $p=1$.
\end{itemize}

\end{itemize}
\end{theorem}

Note that, in case $p=n$, the correction term $\kappa_{\gamma}$ vanishes (indeed it does not appear in the OT characterization of Ricci curvature lower bounds), but for $p<n$ Theorem \ref{Cor:mainSec} is sharp in the sense that one can not suppress $\kappa_{\gamma}$ (see Remark \ref{rem:thmSharp});
more strongly, for the very same example of Remark \ref{rem:thmSharp},  all the inequalities involved in the proof Theorem \ref{Cor:mainSec}  become identities (see Remark \ref{rem:thmisharp}), showing the sharpness of the arguments.

Theorem \ref{Cor:mainSec} is actually a particular case of Theorem \ref{th:lower} (see also Remark \ref{rem:p12nn-1}, for the link between $p$-Ricci and sectional curvatures) where we characterize lower bounds on the $p$-Ricci curvature in terms of optimal transport, 
for any $p\in \{1,\ldots,n\}$. For the rigorous  definition and basic properties of the $p$-Ricci curvature we refer to Section \ref{sec:Prel},  here let us just mention the intuitive idea behind: in the standard Ricci curvature 
(corresponding  in this notation to the $n$-Ricci curvature), one considers the trace of the Riemann curvature tensor along \emph{all the tangent space} to $M$  at some point $x\in M$,   while in the $p$-Ricci curvature one considers the trace of the Riemann curvature tensor \emph{just along  $p$-dimensional subspaces}. The notion of $p$-Ricci curvature has already been considered in the literature, in particular in connection with topological results (see for instance  the works of  Wu \cite{Wu}, Shen \cite{Shen}, Wilhelm \cite{Wilhelm}, Petersen-Wilhelm \cite{PW} and Xu-Ye \cite{XY}). Just to fix the ideas, let us recall that if the sectional curvature is bounded below by $K\geq 0$, then the $p$-Ricci curvature is bounded below by $(p-1)K$; if instead  the sectional curvature is bounded below by $K\leq 0$, then the $p$-Ricci curvature is bounded below by $\min\{p, n-1\} K$.
\\
\medskip

The paper is organized as follows:  Section \ref{sec:Prel} settles the notation and the  preliminaries. In Section  \ref{sec:W2Rect} we analyze $\cH^{p}$-rectifiable $W_{2}$-geodesics and in Section \ref{sec:Jacobi} we perform the Jacobi fields computations/estimates that will be used to prove the main results. Section \ref{sec:UB} is devoted to the proof of Theorem \ref{thm:UB}, namely the optimal transport characterization of sectional curvature upper bounds. Finally, in Section  \ref{sec:LB} we state and prove our main results characterizing sectional and $p$-Ricci lower bounds in terms of optimal transportation; as a consequence, we also obtain a new Brunn-Minkowski type inequality involving $p$-dimensional submanifolds and the $p$-Ricci curvature (see Corollary \ref{cor:BM}).  
\\
\bigskip

{\bf Acknowledgement}:   Most of the work was done while both  authors were in residence at the 
Mathematical Sciences Research Institute in Berkeley, California during the 
Spring 2016 semester,  supported by the National Science Foundation
under Grant No. DMS-1440140.  We thank the organizers of the Differential Geometry Program and  MSRI for providing great environment for research and collaboration.
\\In the final steps of the project, A. M. has been supported by the EPSRC First Grant  EP/R004730/1.  
\\We also wish to express our gratitude to Robert McCann for suggesting the  Remark \ref{rem:rectmut}, and to  Martin Kell and  Gerardo Sosa  for  their careful reading  of the manuscript.

\section{Preliminaries}\label{sec:Prel}
\paragraph{\textbf{Optimal transport and Wasserstein geometry}}
It is out of the scopes of this short section to give a comprehensive introduction to  optimal transport, for this purpose we refer to   \cite{viltot}. 
Instead, we will be  satisfied by recalling those notions and results that we will use throughout the paper. 

Let $(X,\sfd)$ be a complete, separable and proper metric space.
A curve $\gamma:[0,1]\rightarrow X$ is said to be a (length-minimizing, constant speed) geodesic  if 
$$
\sfd(\gamma(s), \gamma(t))= |s-t| \, \sfd(\gamma(0),\gamma(1)), \quad \forall s,t \in [0,1].
$$
We denote by $\Geo(X):=\{\gamma: [0,1]\to X \text{ s.t. } \gamma \text{ is a geodesic}\}$ the family of 
geodesics  equipped with the $L^{\infty}$-topology. The evaluation map $\ee_t:\Geo(X)\rightarrow X$ is given by $\ee_t(\gamma)=\gamma(t)$, 
and it is  clearly continuous with respect to the sup-distance $\sfd_{\infty}(\gamma,\tilde{\gamma})=\sup_{t\in [0,1]}\sfd(\gamma(t),\tilde{\gamma}(t))$.

$\mathcal{P}_{c}(X)$ denotes the space of Borel probability measures with compact support and $\mathcal{P}_{2}(X)$ denotes the space of Borel probability
measures $\mu$ with finite second moment, i.e. satisfying $\int_{X} \sfd^{2}(x, x_{0}) \, d\mu(x)<\infty$ for some (and thus for any) $x_{0}\in X$. 

The space $\mathcal{P}_{2}(X)$ is naturally endowed with the $L^2$-Wasserstein distance $W_{2}$ defined by
$$
W_{2}(\mu_{1}, \mu_{2})^{2}:=\inf\left\{\int_{X\times X} \sfd^{2}(x,y) d\pi(x,y) \, \text{ s.t. } \pi \in  \Cpl(\mu_1,\mu_2)\right\},
$$
where $\Cpl(\mu_1,\mu_2)$ is the family of all couplings between $\mu_1$ and $\mu_2$, i.e. of all the probability measures $\pi\in \mathcal{P}(X^2)$ such that $(P_i)_{\sharp}\pi=\mu_i$, $i=1,2$,   $P_1,P_2$ being the projection maps. $(\mathcal{P}_{2}(X),W_2)$ becomes a separable metric space that is
a geodesic metric space provided $X$ is a geodesic metric space.

A coupling $\pi\in \Cpl(\mu_1,\mu_2)$ is optimal if 
\begin{align*}
\int_{X^2}\sfd(x,y)^2d\pi(x,y)=W_2(\mu_1,\mu_2)^2.
\end{align*}
Optimal couplings always exist, and if an optimal coupling $\pi$ is induced by a map $T:Z\rightarrow X$ via $(T,{\rm Id}_{\sX})_{\sharp}\mu_1=\pi$, where $Z$ is a measurable subset of $X$, we say that $T$ is an optimal map.
A probability measure $\Pi\in \mathcal{P}(\Geo(X))$ is called an optimal dynamical coupling or plan if $(\ee_0,\ee_1)_{\sharp}\Pi$ is an optimal coupling between the initial and final marginal distribution. 
For every $W_{2}$-geodesic $\{\mu_{t}\}_{t\in [0,1]}$ there exists  an optimal dynamical plan $\Pi\in \mathcal{P}(\Geo(X))$ such that $\mu_{t}=(\ee_{t})_{\sharp} \Pi$ for all $t\in [0,1]$.

\medskip

In the present paper, a key role is played by the subspace $\mathcal{P}_2(X,\mathcal{H}^p)\subset \mathcal{P}_2(X)$ made of  probability  measures that are absolutely continuous with respect to the $p$-dimensional Hausdorff measure
$\mathcal{H}^p$. We also denote with  $\mathcal{P}_{c}(X,\mathcal{H}^p)\subset \mathcal{P}_2(X,\mathcal{H}^p)$  the subspace of absolutely  continuous probability measures with compact support. 
\medskip

In the introduction, for simplicity,  we defined the entropy functionals for compactly supported probability  measures; the definitions carry over to probability  measures with  finite second moment, let us briefly recall  them.
Given $1\leq p\leq p'<\infty$, the $p'$-Renyi entropy with respect to the $p$-dimensional Hausdorff measure $\mathcal{H}^p$ is defined as
\begin{align*}
S_{p'}(\cdot|\mathcal{H}^p):\mathcal{P}_2(X,\mathcal{H}^p) \rightarrow [-\infty,0],\ \ S_{p'}(\mu|\mathcal{H}^p)=-\int\rho^{1-\frac{1}{p'}}d\mathcal{H}^p,
\end{align*}
where $\rho$ is the density of $\mu$ with respect to $\mathcal{H}^p$, i.e. $\mu=\rho \mathcal{H}^p$. Notice that,  by Jensen's inequality,  we have
\begin{align*}
[-\infty,0]\ni-\big(\mathcal{H}^p(\supp\mu)\big)^{1/p'}\leq S_{p'}(\mu|\mathcal{H}^p).
\end{align*}
In particular if $\rho$ is concentrated on a set of finite $\mathcal{H}^p$-measure then $S_{p'}(\mu|\mathcal{H}^p)>-\infty$.  
Note that in the borderline case $p'=p=1$, one gets 
$$S_{1}(\mu|\mathcal{H}^{1})=- \mathcal{H}^{1}(\supp(\mu)).$$

Finally, the (relative) Shannon entropy is defined by
\begin{align*}
\Ent(\cdot|\mathcal{H}^p):\mathcal{P}_2(X,\mathcal{H}^p) \rightarrow [-\infty,\infty],\ \ \Ent(\mu|\mathcal{H}^p)= \lim_{\varepsilon\downarrow 0} \int_{\{\rho>\varepsilon\}} \rho  \log \rho \, d\mathcal{H}^p.
\end{align*}
This coincides with  $ \int_{\{\rho>0\}} \rho  \log \rho \, d\mathcal{H}^p$, provided that $ \int_{\{\rho\geq 1\}} \rho  \log \rho \, d\mathcal{H}^p < \infty$, and $\Ent(\mu|\mathcal{H}^p):=\infty$ otherwise.

\paragraph{\textbf{Rectifiable sets}} 
Let $\Sigma\subset \R^{n}$ and $m \in \N, m\leq n$. We say that $\Sigma$ is \emph{countably $m$-rectifiable} if there is a countable family of
Lipschitz maps $f_i: \R^{m} \to \R^{n}$,  such that $\Sigma \subset \bigcup_{i}f_i(\R^{m})$.
The set $\Sigma$ is \emph{countably $\mathcal{H}^m$-rectifiable} if there is a {countably $m$-rectifiable set} $\Sigma'\subset \R^{n}$ such that $\mathcal{H}^m(\Sigma \backslash \Sigma')=0$.

As it is well known, using Whitney extension Theorem, it is possible to show that  a subset $\Sigma\subset \mathbb{R}^n$ is countably $\mathcal{H}^m$-rectifiable if and only if there exists a sequence of $m$-dimensional $C^1$-submanifolds $\{S_i\}_{i\in\N}$ such that 
$$\mathcal{H}^m\Big(\Sigma \backslash \bigcup_{i\in \N} S_i\Big)=0.$$

\noindent Clearly, by considering local coordinates (or by Nash isometric embedding Theorem),  one can define the same notions for  subsets of an $n$-dimensional Riemannian manifold.
\medskip
\paragraph{\textbf{Intermediate Ricci curvature}}
Let $(M,g)$ be an $n$-dimensional  Riemannian manifold and let
$$R: TM\times TM \times TM \to TM,  \quad R(X,Y)Z:=\nabla_{Y} \nabla_{X}Z-\nabla_{X}\nabla_{Y}Z+\nabla_{[X,Y]}Z$$
 the Riemannian curvature tensor (of course $\nabla$ denotes the Levi-Civita connection of $(M,g)$ and $[\cdot, \cdot]$ denotes the Lie bracket). Sometimes we will use the notation $|v|:=\sqrt{g(v,v)}$ and $\langle v, w\rangle:=g(v,w)$. Using the standard notation, $T_{x}M$ is the tangent space of $M$ at the point $x \in M$.
  For a $2$-plane $P\subset T_{x}M$ spanned by  $v,w\in T_{x}M$, let  
  $$\sect(P)=\sect(v,w):=\frac{\langle R(v,w)v, w\rangle}{|v|^{2} |w|^{2}- \langle v, w \rangle^{2}}$$
   be the sectional curvature. Recall that, given $w \in T_{x}M$, the Ricci curvature $\ric(w,w)$ is defined by
   $$\ric(w,w):=\mbox{tr}\left[R(w,\cdot)w \right]. $$
 
  \begin{definition}[$p$-Ricci Curvature]
Let $p\in \{1,\ldots, n\}$.   For a $p$-dimensional plane $P$ in $T_{x}M$ and a vector $w\in T_{x}M$, we define the $p$-Ricci curvature of $P$ in the direction of $w$ as
\begin{align}\label{eq:defRicp}
\ric_p(P,w)&:=\mbox{tr}\left[\top_{P}\circ\big(R(w,\cdot)w \big)|_P\right] =\sum_{i=1}^{p} \sect(e_i,w)(|w|^2-\langle e_i,w\rangle^2),
\end{align}
where $e_1,\dots,e_p$ is an orthonormal basis of $P$, and $\top_P:T_{x}M\rightarrow P$ is the orthogonal projection of $T_{x}M$ onto $P$.
\end{definition}

Note that, in particular, if $|w|=1$ and $w$ is orthogonal to $P$ then  
$$\ric_p(P,w)=\sum_{i=1}^{p}\sect(e_i,w).$$
It is standard to check that $\ric_p$ is well-defined and 
independent of the choice of a basis for $P$.  Notice also that, 
if $w\notin P$, then 
\begin{align}\label{eq:ricpp+1}
\ric_p(P,w)&=\ric_{p+1}(\mbox{span}(P,w),w)=\ric_p(\mbox{span}(P,w)\cap w^{\perp},w) = \sum_{i=1}^{p} \sect(e_{i}, w) |w|^{2},
\end{align}
where $\{e_i\}_{i=1,\dots,p}$ is an orthonormal basis of $\mbox{span}(P,w)\cap w^{\perp}$, $w^{\perp}\subset T_{x}M$ being the orthogonal  subspace to $w$.

\begin{definition}[$p$-Ricci upper and lower bounds]
We say that $(M,g)$ has $p$-Ricci curvature bounded from below (resp. from above) by $K$ if, for any $x \in M$ and  any $p$-dimensional plane $P\subset T_{x}M$, we have $\ric_p(P,w)\geq K|w|^2$ (resp. $\ric_p(P,w)\leq K|w|^2$); in this case we write $\ric_{p}\geq K$ (resp. $\ric_{p}\leq K$).
\end{definition}

\begin{remark}[Some notable cases]\label{rem:p12nn-1} 
The cases $p=1,2$ are strictly linked with the sectional curvature while $p=n-1,n$ are related to the standard Ricci curvature. More precisely 
\begin{itemize}
\item $p=1$: if $P$ is the real line spanned by $v$, $\langle v, w \rangle=0$, $|v|=|w|=1$, then
$$ \ric_1(P,w)=\sect(v,w);$$ 
on the other hand  $\ric_{1}(P,v)=0$, i.e. the 1-Ricci curvature always vanishes in the  direction of $P$ itself. 
In particular no Riemannian manifold has  $1$-Ricci curvature bounded from below (resp. above) by a strictly positive (resp. negative) constant. 
Nevertheless $M$ has non-negative (resp. non-positive) 1-Ricci curvature if   and only if the sectional curvature is non-negative (resp. non-positive).

\item  $p=2$:  if $P$ is the 2-plane spanned by the orthonormal vectors $e_{1},e_{2}$ then
\begin{equation}\label{eq:ric2sec}
 \ric_2(P,e_{1})=\ric_2(P,e_{2})=\sect(e_{1},e_{2}).
 \end{equation}
Moreover, if $w$ is orthogonal to $P$ with $|w|=1$ then
$$ \ric_2(P,w)=\sect(e_{1}, w)+\sect(e_{2},w).$$
In particular for every $K\geq 0$ (resp. $K\leq 0$),  it holds $\ric_{2}\geq K$ (resp. $\ric_{2}\leq K$) if and only if $\sect \geq K$ (resp. $\sect \leq K$). Note also that if $\sect \geq K\geq 0$ then for every $p\in\{2,\ldots, n\}$ it holds $\ric_{p}\geq (p-1)K$.
\item $p=n-1$:  if $P$ is an $n-1$-plane and $w$ is orthogonal to $P$, then 
$$\ric_{n-1}(P,w)=\ric(w,w). $$

\item $p=n$: in this case one has $P=T_{x}M$, and for every $w\in T_{x}M$ it holds
 $$\ric_n(T_x M,w)=\ric(w,w).$$  
 
 \item If $\sect \geq K$, depending on the sign of $K \in \R$ we have:
 \begin{itemize}
 \item[$\cdot$]  $\sect \geq K \geq 0$ implies that $\ric_{p}\geq (p-1) K$, for all $p\in \{1, \ldots, n\}$
 \item[$\cdot$]   $\sect \geq K$ with $K\leq 0$ implies that $\ric_{n}\geq (n-1) K$ and  $\ric_{p}\geq p K$ for  all $p\in \{1, \ldots, n-1\}$.   
 \end{itemize}
\end{itemize}
\end{remark}

\section{Countably $\mathcal{H}^p$-rectifiable geodesics in Wasserstein space} \label{sec:W2Rect}
 The next result is a well known consequence of the Monge-Mather shortening principle \cite[Theorem 8.5]{viltot}.
\begin{theorem}\label{thm:MongeMather}
Consider a Riemannian manifold $(M,g)$, fix a compact subset  $E \subset \subset M$ and let $\Pi$ be a dynamical optimal plan such that $(\ee_{t})_{\sharp}\Pi$ is supported in $E$ for every $t \in [0,1]$. 

Then $\Pi$ is supported on a set of geodesics $S\subset \Geo(M)$ satisfying the following: for every $t_{0}\in (0,1)$ there exists $C_E(t_{0})>0$ such that for any two geodesics $\gamma,\eta \in S$ it holds
\begin{align*}
\sup_{t \in [0,1]} \sfd(\gamma(t), \eta(t)) \leq C_E(t_{0}) \sfd(\gamma(t_{0}), \eta(t_{0})),
\end{align*}
where $\sfd$ is the Riemannian distance on $(M,g)$. 
\end{theorem}

\begin{remark}\label{rem:MM}
As a consequence of Theorem \ref{thm:MongeMather}, if $\{\mu_t\}_{t\in[0,1]}$ is a $W_{2}$-geodesic such that $\mu_{0},\mu_{1}$ are compactly
supported probability measures on $M$, and $t_{0}\in (0,1)$ is given, then for any $t\in [0,1]$ the
map
$T_{t_{0}}^{t}:\gamma(t_{0})\mapsto\gamma(t)$ is well-defined $\mu_{t_{0}}$-almost everywhere and Lipschitz continuous on its domain; moreover  it is
the unique optimal transport map between $\mu_{t_{0}}$ and $\mu_t$. In other words, the optimal coupling $(\ee_{t_{0}},\ee_t)_{\sharp}\Pi$ is induced by $T_{t_{0}}^{t}$, i.e.  $(\ee_{t_{0}},\ee_t)_{\sharp}\Pi=({\rm Id}, T_{t_{0}}^{t})_{\sharp} \mu_{t_{0}}$.
\end{remark}

\begin{lemma}\label{rem:mutPreiss}
Let $(M,g)$ be a  complete $n$-dimensional Riemannian manifold without boundary, and let $p\in \{1, \ldots, n\}$. Let $\mu_0,\mu_1\in \mathcal{P}_{c}(M,\mathcal{H}^p)$ and assume $\{\mu_t\}_{t \in [0,1]}$ is a $W_{2}$-geodesic  between $\mu_0, \mu_1$ such that for some $t_{0}\in (0,1)$ the measure $\mu_{t_{0}}$ is concentrated on a countably $\cH^{p}$-rectifiable set $\Sigma_{t_{0}}\subset M$. 

 Then  for every  $t \in [0,1]$ there exists a   countably $\cH^{p}$-rectifiable set $\Sigma_{t}\subset M$ such that $\mu_{t}$ is concentrated on $\Sigma_{t}$; moreover   $\mu_{t}=\rho_{t} \cH^{p} \llcorner \Sigma_{t}\in \mathcal{P}_{c}(M,\mathcal{H}^p)$  for  a suitable probability density $\rho_{t}\in L^{1}(M, \cH^{p})$.   
\end{lemma}

\begin{proof} 
\textbf{Step 1.}
By Theorem \ref{thm:MongeMather} and Remark \ref{rem:MM} we know that for every $t \in [0,1]$ there exists a Lipschitz map $T_{t_{0}}^{t}:\supp \mu_{t_{0}} \to \supp \mu_{t}$ such that $\mu_{t}=(T_{t_{0}}^{t})_{\sharp} \mu_{t_{0}}$. Since by assumption $\mu_{t_{0}}$ is concentrated on the countably $\cH^{p}$-rectifiable set $\Sigma_{t_{0}}$, it is then clear that $\mu_{t}$ is concentrated on $\Sigma_{t}:= T_{t_{0}}^{t} (\Sigma_{t_{0}})$ which is  countably $\cH^{p}$-rectifiable set too, as Lipschitz image of a countably $\cH^{p}$-rectifiable set. In order to conclude that $\mu_{t}=\rho_{t}\, \cH^{p}\llcorner \Sigma_{t}\in \mathcal{P}_{c}(M,\mathcal{H}^p)$ it is then enough to show that $\mu_{t}(A)=0$ for every   $A\subset \supp \mu_{t}$ satisfying $\mathcal{H}^p(A)=0$. 
This will be proved in Step 3, using the discussion of Step 2.
\\

\textbf{Step 2.}
Consider $\mu_0,\mu_1\in \mathcal{P}_c(M,\mathcal{H}^p)$ and write $\mu_i=\rho_i \mathcal{H}^p$ for $i=0,1$.
Let $\Pi\in \mathcal{P}(\Geo(M))$ a dynamical optimal plan between $\mu_0$ and $\mu_1$, and let $\left\{\mu_{t}:=(\ee_t)_{\sharp}\Pi\right\}_{t\in[0,1]}$ be the induced $L^2$-Wasserstein geodesic. 
We denote with  $\pi_{t,s}=(\ee_t,\ee_s)_{\sharp}\Pi$ the corresponding optimal coupling between $\mu_t$ and $\mu_s$ for any $t,s\in[0,1]$. Since $\mu_0$ and $\mu_1$ have compact support, then there exists  a compact subset $E\subset \subset M$ such that $\supp \mu_{t} \subset E$ for every $t \in [0,1]$.

By Theorem \ref{thm:MongeMather}, the dynamical optimal plan $\Pi$ is supported on set $S\subset \Geo(M)$ satisfying the following: for any  $t\in (0,1)$  there exists  $C_E(t)>0$ such that for any $s\in [0,1]$   it holds 
\begin{align}
\sfd(\gamma(s),\eta(s))\leq C_E(t)\sfd(\gamma(t),\eta(t)) \ \ \mbox{for any pair }\gamma,\eta \in S.
\end{align}
As observed  in Remark \ref{rem:MM}, the optimal plan $\pi_{t,s}$ is then induced  by a Lipschitz-continuous  optimal  transport map $T_{t}^{s}:\supp \mu_t \to \supp \mu_{s}$ with Lipschitz constant bounded above by $C_E(t)$.  In particular   $(T_{t}^{s})_\sharp \mu_t=\mu_s$. 
\\

\textbf{Step 3.} Let $t \in (0,1)$, and consider $\pi_{t,0}:= (\ee_t,\ee_0)_{\sharp}\Pi$. 
Our goal is to show that  if $A\subset \supp \mu_{t}$ satisfyies $\mathcal{H}^p(A)=0$, then $\mu_{t}(A)=0$ as well. Since by Step 2 the plan $\pi_{t,0}$ is induced by the map $T_{t}^{0}$,    we have
\begin{equation}\label{eq:musA}
\mu_t(A)=\pi_{t,0}(A,M)=\pi_{t,0}(A, T_{t}^{0}(A)).
\end{equation}
On the other hand 
\begin{equation}\label{eq:mu0Ts0A}
\pi_{t,0}(A, T_{t}^{0}(A))\leq  \pi_{t,0}(M, T_{t}^{0}(A)) =\mu_0(T_{t}^{0}(A)). 
\end{equation}
Since $T_{t}^{0}:\supp \mu_{t} \to \supp \mu_{0}$ is Lipschitz and $\cH^{p}(A)=0$, then it  also holds $\cH^{p}(T_{t}^{0}(A))=0$.  Recalling that  by assumption $\mu_{0}\ll \cH^{p}$,  we then  get that $\mu_0(T_{t}^{0}(A))=0$. 
The claim follows then by the combination of \eqref{eq:mu0Ts0A} and \eqref{eq:musA}.
\end{proof}

\begin{definition}\label{def:HpRectW2Geo}
We say that a  $W_2$-geodesic $\{\mu_t\}_{t\in [0,1]}$ is \emph{countably $\mathcal{H}^p$-rectifiable} if for every $t\in [0,1]$ the measure $\mu_{t}\in  \mathcal{P}_c(M,\mathcal{H}^p)$ is concentrated on a countably $\mathcal{H}^p$-rectifiable set $\Sigma_{t}\subset M$.
\end{definition}

\begin{remark}\label{rem:mutRect}
By Lemma \ref{rem:mutPreiss}, a  $W_2$-geodesic $\{\mu_t\}_{t\in [0,1]}$ is countably $\mathcal{H}^p$-rectifiable if and only if $\mu_0,\mu_1\in \mathcal{P}^2_c(M,\mathcal{H}^p)$ and there exists $t_{0}\in (0,1)$ such that the measure $\mu_{t_{0}}$ is concentrated on a countably $\mathcal{H}^p$-rectifiable set $\Sigma_{t_{0}}\subset M$.
\end{remark}

\begin{remark}\label{rem:welldefDet}
 Note that, in Definition \ref{def:HpRectW2Geo},  one can replace   $\Sigma_{t}$ by $\Sigma_{t}\cap \supp {\mu_{t}}$; thus from now on we will always tacitly assume that $\Sigma_{t}=\Sigma_{t}\cap \supp {\mu_{t}}$, for all $t \in [0,1]$.
Also,  since for $s\in (0,1)$ and $t\in [0,1]$ the optimal transport map  $T_{s}^{t}$ given in Remark \ref{rem:MM}  is well defined  $\mu_{s}$-a.e.,  from now on we will just consider the restriction  $T_{s}^{t}\llcorner \Sigma_{s}$ 
and, for simplicity of notation,  write $T_{s}^{t}$ to indicate the  map $T_{s}^{t}\llcorner \Sigma_{s}: \Sigma_{s}\to T_{s}^{t}(\Sigma_{s})$. 
 Note that, with this notation, for $\mu_{s}$-almost every $x$, the differential $DT_{s}^{t}(x)$ is a linear map from the 
 $p$-dimensional space $T_{x}\Sigma_{s}$ to the $q$-dimensional space $T_{T_{s}^{t}(x)} (T_{s}^{t}(\Sigma_{s}))$, $q\leq p$ ($q$ possibly depending on $x$).
 \end{remark}
 
\begin{remark}[A sufficient condition for the $p$-rectifiability of $\mu_{t}$]\label{rem:rectmut}
The following sufficient condition for the $p$-rectifiability of the geodesic $\mu_{t}$ follows by combining the work of   McCann-Pass-Warren  \cite[Theorem 1.2]{MPW} with Lemma \ref{rem:mutPreiss}.

Given $p \in \{1,\ldots, n\}$, let   $\mu_0,\mu_1\in \mathcal{P}_{c}(M,\mathcal{H}^p)$ with $\mu_i=\rho_i \,\mathcal{H}^p \llcorner \Sigma_{i}$, for some smooth $p$-dimensional submanifolds  $\Sigma_{i}$, $i=1,2$.
 Consider the restriction of the quadratic cost function $\sfd^{2}$ to the product $\Sigma_{0}\times \Sigma_{1}$; if 
 \begin{equation}\label{eq:HessND}
 \det\left[\left(\frac{\partial^2}{\partial x_i\partial y_j}\sfd^{2}|_{\Sigma_{0}\times \Sigma_{1}}\right)_{i,j=1,\dots,p} \right]\neq 0
 \end{equation}
and moreover
\begin{equation}\label{eq:CutLoc}
\Sigma_{0} \cap \big( \bigcup_{x \in \Sigma_{1}} \textrm{Cut}(x) \big) =\emptyset \quad \text{and} \quad  \Sigma_{1} \cap \big( \bigcup_{x \in \Sigma_{0}} \textrm{Cut}(x) \big) =\emptyset,
\end{equation} 
where $ \textrm{Cut}(x) $ is the cut locus of the point $x\in M$,   then every $W_{2}$-geodesic   $\{\mu_t\}_{t \in [0,1]}$ between $\mu_0$ and $\mu_1$ satisfies  that   $\mu_{t}=\rho_{t} \cH^{p} \llcorner \Sigma_{t}\in \mathcal{P}_{c}(M,\mathcal{H}^p)$ for every $t\in [0,1]$ for a  countably $\cH^{p}$-rectifiable set $\Sigma_{t}\subset M$.

Indeed,  calling $\pi$ the $L^{2}$-optimal coupling induced by the geodesic $\{\mu_t\}_{t \in [0,1]}$,  by using \eqref{eq:HessND}  we can apply \cite[Theorem 1.2]{MPW} and get that  $\pi$ is supported on a  $p$-dimensional Lipschitz submanifold $S$ of $ \Sigma_{0}\times \Sigma_{1} \subset M\times M$.
Using now \eqref{eq:CutLoc}, we get that for every $(x,y)\in \Sigma_{0}\times \Sigma_{1}$ there exists a unique geodesic $t\mapsto \gamma_{t}(x,y)$  from $x=\gamma_{0}(x,y)\in \Sigma_{0}$ to $y=\gamma_{1}(x,y) \in \Sigma_{1}$; moreover the map $\gamma_{t}(\cdot,\cdot):  \Sigma_{0}\times \Sigma_{1}\to M\times M$ is Lipschitz, for every fixed $t\in [0,1]$.  Calling $\Sigma_{t}:=\gamma_{t}(S)$ we get that $\mu_{t}$ is concentrated on the $p$-rectifiable subset $\Sigma_{t}$. 
The fact that we can write $\mu_{t}= \rho_{t} \cH^{p} \llcorner \Sigma_{t}$ for some density  $\rho_{t}\in L^{1}(\cH^{p})$ follows then by Lemma \ref{rem:mutPreiss}. 
\end{remark}

\begin{remark}
 For fixed $s$ and $t$,  pick a (resp. orthonormal) basis $(e_i)_{i=1,\dots,p}$ of $T_x\Sigma_s\subset T_xM$,  
 and also a (resp. orthonormal) basis $(f_i)_{i=1,\dots,n}$ of  $T_{T_{s}^{t}(x)}M$ such that $(f_i)_{i=1,\dots,q}$ is a basis of $T_{T_s^t(x)}(T_s^t(\Sigma_s))$.  
 We can then see $DT_{s}^{t}(x)$ as a linear map from $\R^{p}$ to $\R^{p}$ (if $q<p$ just identify $\R^{q}$ with $\{(x^{1}, \ldots, x^{p}) \, : \, x^{1}=\ldots=x^{p-q}=0\}$). Since  the rank and the  determinant are independent of the chosen basis,  $\det [DT_{s}^{t}(x)]$ and   the fact that  $DT_{s}^{t}(x)$ is non-degenerate are then well defined concepts.
 \end{remark}
In the next lemma we show that the optimal transport map $T_{s}^{t}$ is differentiable $\mu_{s}$-a.e. on $\Sigma_{s}$ and that at least an inequality holds in the Monge-Amp\`ere equation (cf. \cite{CMS}); this will be sufficient (and crucial) to our aims of  characterizing curvature bounds in terms of optimal transport.

\begin{lemma}\label{B}
Let $M$ be a complete Riemannian manifold and $\{\mu_{t}\}_{t\in[0,1]}$  a $W_{2}$-geodesic with $\mu_{t}\ll \cH^{p}\llcorner \Sigma_{t} \in {\mathcal P}_{c}(M,\cH^{p})$ for some countably $\cH^{p}$-rectifiable subset $\Sigma_{t}\subset M$, for every $t\in [0,1]$.  For fixed $s\in(0,1)$ and $t \in [0,1]$, let $T_{s}^{t}$ be the optimal transport map from $\mu_{s}$ to $\mu_{t}$ given in Remark \ref{rem:MM}. 
 
Then $T_{s}^{t}: \Sigma_{s} \to T_{s}^{t}(\Sigma_{s})\subset  M$ is differentiable $\mu_s$-a.e.  and the following Monge-Amp\`ere inequality  holds:
\begin{equation}\label{eq:MAineq}
\rho_s(x)\leq \det [DT_{s}^{t}(x)] \; \rho_t(T_{s}^{t}(x)) \ \ \mbox{ $\mu_s$-a.e. }  x, \; \forall s\in(0,1),  \, \forall t \in [0,1].
\end{equation}
 In particular, $DT_{s}^{t}:\R^{p}\to \R^{p}$ is $\mu_s$-a.e. non-degenerate. 
Moreover \eqref{eq:MAineq} holds with equality if $t,s\in (0,1)$.\end{lemma}

Let us stress that  in the above lemma we do not claim that $T_{s}^{t}$ is $\mu_{s}$-a.e. differentiable as  a map from $M$ to $M$, but just as a map from $\Sigma_{s}$ to its image, i.e. we claim differentiability with respect to infinitesimal variations which are \emph{tangential} to $\Sigma_{s}$.

\begin{proof}
\textbf{Step 1.} Differentiabiliy $\mu_{s}$-a.e..
\\From Theorem \ref{thm:MongeMather}  and  Remark \ref{rem:MM}, we know that $T_{s}^{t}: \Sigma_{s}\to T_{s}^{t}(\Sigma_{s})$ is a Lipschitz map; since by assumption   $\Sigma_{s}$ is countably $\cH^{p}$-rectifiable, Rademacher Theorem implies that $T_{s}^{t}: \Sigma_{s}\to T_{s}^{t}(\Sigma_{s})$ is differentiable $\cH^{p}$-a.e. . \medskip

\textbf{Step 2.}  Monge-Amp\`ere inequality.
\\Since by construction $(T_{s}^{t})_{\sharp}\mu_{s}=\mu_{t}$, it follows that for an arbitrary Borel subset $A\subset \Sigma_{s}$ it holds
\begin{equation}\label{eq:musleqmut}
\mu_{s}(A)\leq \mu_{s}\left((T_{s}^{t})^{-1}(T_{s}^{t}(A)) \right)=\mu_{t}\left(T_{s}^{t}(A)\right). 
\end{equation}
Equality holds for $s,t \in (0,1)$ as the map $T_{s}^{t}$ is $\mu_{s}$-essentially injective. Recalling that $\mu_{s}=\rho_{s} \, \cH^{p}\llcorner \Sigma_{s}$ and $\mu_{t}=\rho_{t} \, \cH^{p} \llcorner \Sigma_{t}$,  by the area formula we infer that 
\begin{align}
\mu_{t}\left(T_{s}^{t}(A)\right)&=\int_{T_{s}^{t}(A)} \rho_{t} \, d \cH^{p}\llcorner \Sigma_{t} \nonumber \\
& \leq  \int_{T_{s}^{t}(A)}\rho_t(y)\mathcal{H}^0((T_{s}^{t}{\llcorner A})^{-1}(y))\, d\mathcal{H}^p\llcorner \Sigma_{t}(y) =\int_A\rho_t(T_{s}^{t}(x))\det\left[ DT_{s}^{t}(x) \right] \,d\mathcal{H}^p \llcorner \Sigma_{s}(x), \label{eq:AFmusmut}
\end{align}
with equality if $s,t\in (0,1)$  as the map $T_{s}^{t}$ is $\mu_{s}$-essentially injective.  The combination of  \eqref{eq:musleqmut} and \eqref{eq:AFmusmut}  gives that for an arbitrary Borel subset $A\subset \Sigma_{s}$ it holds
$$
\int_{A} \rho_{s}\, d\cH^{p}=\mu_{s}(A)\leq \mu_{t}(T_{s}^{t}(A))\leq \int_A\rho_t(T_{s}^{t})\det\left[DT_{s}^{t}\right] \,d\mathcal{H}^p,
$$
and  the Monge-Amp\`ere inequality \eqref{eq:MAineq} follows, with equality for $s,t\in (0,1)$.
\end{proof}

In order to have a more clear notation, in the next lemma we pick $s=1/2$  and consider the Lipschitz map $T^{t}_{1/2}: \Sigma_{1/2} \to \Sigma_{t}$, $t\in [0,1]$, but the same arguments hold for any fixed $s\in (0,1)$.
For $\mu_{1/2}$-a.e. $x \in M$ let  $\gamma_{x}\in \Geo(M)$ be the geodesic defined by  $[0,1]\ni t  \mapsto \gamma_{x}(t):=T^{t}_{1/2}(x)$ and,
for $\mu_{1/2}$-a.e. $x \in \Sigma_{1/2}$ let $v(x)\in T_{x}M$ be such that $\gamma_{x}(t)=\exp_{x}((t-\frac{1}{2})\, v(x))$, that is $v(x)=\dot{\gamma}_x(\frac{1}{2})$.   Denote also with $\nabla_t$ the covariant derivative in $M$ in the direction of $\dot{\gamma }_{x}(t)$  and with   $D_t=\top \circ \nabla_t $ where  $\top:T_{\gamma_{x}(t)} M \to T_{\gamma_{x}(t)}\Sigma_{t}$ is
 the orthogonal projection map.

\begin{lemma}\label{lem:vDiff}
The map $M\ni x \mapsto v(x)\in TM$ is  well defined and differentiable $\mu_{1/2}$-a.e.\ . As a consequence we can find a subset $N\subset \Sigma_{1/2}$ 
(independent of $t \in [0,1]$ ) with $\mu_{1/2}(N)=0$, such that for every $t \in [0,1]$  the map $T^{t}_{1/2}: \Sigma_{1/2}\to \Sigma_{t}$ is differentiable at  every $x \in \Sigma_{1/2}\setminus N$. 
\\Moreover, up to replacing $N$ with a larger set of null $\mu_{1/2}$-measure,  for every  $x \in \Sigma_{1/2}\setminus N$ the map $t\mapsto D_{x}T^{t}_{1/2}: T_{x}\Sigma_{1/2}\to T_{T^{t}_{1/2}(x)}\Sigma_{t}$ is differentiable at $t=1/2$ and $D_{t}|_{t=1/2} D_{x}T^{t}_{1/2}: T_{x}\Sigma_{1/2}\to T_{x}\Sigma_{1/2}$ is self-adjoint.
\end{lemma}

\begin{proof}

\textbf{Step 1}: the map $v(\cdot):\Sigma_{1/2}\setminus N\to TM$ is well defined.
\\In a first instance let $N\subset \Sigma_{1/2}$, with $\mu_{1/2}(N)=0$, be such that for  every $x \in \Sigma_{1/2}\setminus N$  the curve  $t \mapsto \gamma_{x}(t):=T^{t}_{1/2}(x)$ is 
a well defined geodesic.  In particular,  the curve $t \mapsto \gamma_{x}(t)$ is $C^{1}$ and we can set $v(x)= \dot{\gamma}_{x}(\frac{1}{2})$; this is clearly   well defined as a  map from $\Sigma_{1/2}\setminus N$ to $TM$. Note moreover that, since by standing assumption $\mu_{0}$ and $\mu_{1}$ (and therefore all the measures $\mu_{t}$) have compact support,  we have 
\begin{equation}\label{eq:vbounded}
 \mu_{1/2}{\rm-ess\, sup}_{x \in \Sigma_{1/2}} |v(x)| \leq \frac{1}{2}   \sup_{(x,y) \in \supp \mu_{0} \times \supp {\mu_{1}}} \sfd(x,y) =: C_{\mu_{0},\mu_{1}} < \infty. 
 \end{equation}

\textbf{Step 2}: the map $v(\cdot):\Sigma_{1/2}\setminus N\to TM$ is differentiable.
\\First of all, note that  there exists $\delta>0$ small enough so that $T_{x}M \supset B_{C_{\mu_{0},\mu_{1}}}(0)  \ni w \mapsto \exp_{x}(tw)$ is a diffeomorphism onto its image for
 every $t \in (-\delta,\delta)$ and every $x \in \Sigma_{1/2}$.
 Fix $x_{0}\in \Sigma_{1/2}\setminus N$.  Since by Lemma  \ref{B} the map $T_{1/2}^{(1+\delta)/2}$ is differentiable $\mu_{1/2}$-a.e.,  it follows that also the map
\begin{equation*}
B_{\delta C_{\mu_{0}, \mu_{1}}}(x_{0})\cap \Sigma_{1/2} \setminus N \to TM, \qquad x \mapsto v(x):= \frac{2}{\delta} \exp_{x}^{-1} \left(T_{1/2}^{(1+\delta)/2}(x)\right)
\end{equation*}
is differentiable $\mu_{1/2}$-a.e.. Therefore, up to redefining  the $\mu_{1/2}$-negligible set $N$, the claim is proved. Notice that  in particular the map
\begin{equation}\label{eq:diff12}
T_{1/2}^{t}(x)=\exp_{x}(t v(x))\quad   \text{is differentiable everywhere on $\Sigma_{1/2}\setminus N$ for every $t \in \Big(\frac{1-\delta}{2}, \frac{1+\delta}{2}\Big)$}.
\end{equation}

\textbf{Step 3}: the map $T^{t}_{1/2}:\Sigma_{1/2}\setminus N\to \Sigma_{t}$ is differentiable for every $t \in [0,1]$.
\\By construction we have that  $T^{t}_{1/2}(x)=\exp_{x}(t v(x))$ and, using again that $\mu_{0}$ and $\mu_{1}$ have compact support, we know that there exists a compact subset $E\subset \subset M$ such that $T_{1/2}^{t} (\Sigma_{1/2})\subset E$ for every $t \in [0,1]$. In particular, there exists $\delta>0$ small enough such that, for every $x_{0} \in E$,  the exponential map 
$$\exp_{(\cdot)}(\cdot): B_{\delta  C_{\mu_{0},\mu_{1}}}(x_{0})\times B_{\delta C_{\mu_{0},\mu_{1}}}(0)\to M$$
 is smooth,  where $C_{\mu_{0},\mu_{1}}$ was defined in \eqref{eq:vbounded}.
 \\Let $t_{j}:= \frac{1}{2}+ \frac{\delta}{2} j $, for $j=-\lfloor \frac{1}{\delta} \rfloor, \ldots, 0, \ldots,  \lfloor \frac{1}{\delta} \rfloor$, be a $\frac{\delta}{2}$-grid in $[0,1]$ centered at $1/2$; for convenience choose $\delta \notin \Q$ so that $\frac{1}{2}+ \frac{\delta}{2} \lfloor \frac{1}{\delta} \rfloor <1$. By repeating the same argument of step 2 and replacing $1/2$ by $t_{j}$ in \eqref{eq:diff12},  we get that  for every $j=-\lfloor \frac{1}{\delta} \rfloor, \ldots, 0, \ldots,  \lfloor \frac{1}{\delta} \rfloor$ there exists a subset $N_{j}\subset \Sigma_{t_{j}}$ with $\mu_{t_{j}}(N_{j})=0$ such that $T_{t_{j}}^{t}$ is differentiable everywhere on $\Sigma_{t_{j}}\setminus N_{j}$ for every $t \in (t_{j-1}, t_{j+1})$.
 \\Since by Lemma \ref{B} the maps $T_{t_{i}}^{t_{i+1}}: \supp \mu_{t_{i}} \to \supp \mu_{t_{i+1}}$ are bi-Lipschitz and since $\mu_{t}$ is equivalent to $\cH^{p}\llcorner (\Sigma_{t}\cap \{\rho_{t}>0\})$ for every $t \in [0,1]$, we get that
 $$
 N_{+}:= N_{0} \cup \left[ (T_{1/2}^{t_{1}})^{-1} (N_{1}) \right]  \cup  \left[  (T_{t_{1}}^{t_{2}} \circ T_{1/2}^{t_{1}})^{-1} (N_{2}) \right] \cup \ldots  \cup  \left[ ( T^{t}_{ \lfloor \frac{1}{\delta} \rfloor }\circ \ldots \circ  T_{t_{1}}^{t_{2}} \circ T_{1/2}^{t_{1}})^{-1} (N_{ \lfloor \frac{1}{\delta} \rfloor }) \right]
 $$
 satisfies $\mu_{1/2}(N_{+})=0$.  Defining analogously $N_{-}$ by considering $t_{j}\leq \frac{1}{2}$ and setting $N=N_{+}\cup N_{-}$ we get that $\mu_{1/2}(N)=0$. 
 \\Fix now an arbitrary $t \in [1/2,1]$ and  let $j_{0}:=\max\{j \,:\, t_{j} \leq t \}$. Since  we can write $T_{1/2}^{t}=T^{t}_{t_{j_{0}}}\circ \ldots \circ  T_{t_{1}}^{t_{2}} \circ T_{1/2}^{t_{1}}$, it follows that 
 $T_{1/2}^{t}:\Sigma_{1/2}\setminus N\to \Sigma_{t}$ is differentiable everywhere, as composition of differentiable functions. The argument for $t \in [-1, 1/2]$ is completely analogous, so the lemma is proved.
 
  \textbf{Step 4}:  the map $t\mapsto D_{x}T^{t}_{1/2}: T_{x}\Sigma_{1/2}\to T_{T_{t}^{1/2}(x)}\Sigma_{t}$ is differentiable at $t=1/2$ and  $D_{t}|_{t=1/2} D_{x}T^{t}_{1/2}: T_{x}\Sigma_{1/2}\to T_{x}\Sigma_{1/2}$ is self-adjoint.
 
 Recall that the transport geodesics $\gamma_{x}(t):=T_{1/2}^{t}(x)$ are precisely the gradient flow curves of the corresponding Hamilton-Jacobi shift $\phi_t={\rm H}_{t}\phi$ of a $\sfd^{\scriptscriptstyle 2}$-Kantorovich potential $\phi$ from $\mu_{0}$ to $\mu_{1}$. Note in particular that, for $t\in (0,1]$, $\phi_t$ is semi-concave.
Moreover,  $\phi_{1/2}$ is differentiable in $x$ and admits a gradient $\nabla \phi_{1/2}$ in the classical sense with $\nabla \phi_{1/2}(x)= \dot{\gamma}_{x}(1/2)$  (page 61 in \cite{pet}).
In particular, $\phi_{1/2}$ is differentiable $\mu_{1/2}$-a.e.  .

By construction, $\nabla \phi_{1/2}$ coincides  $\mu_{1/2}$-a.e. with the vector field $v$ defined above  and $v$ is
also differentiable on $\Sigma_{1/2}\setminus N$ as vector field along $\Sigma_{1/2}$. For every $x\in \Sigma_{1/2}\setminus N$  denote $A_{x}:T_{x} \Sigma_{1/2}\to T_{x} \Sigma_{1/2}$,  $A_{x}w:= D_wv(x):=(\nabla_{w} v(x))^{\top}$.
For any fixed $x\in \Sigma_{1/2}\setminus N$, thanks to the semi-concavity of $\phi_{1/2}$  and the differentiability of $v:\Sigma_{1/2}\backslash N\rightarrow TM$
we can follow the proof of  \cite[Theorem 2.8]{rock} to show
that the second derivatives of $\phi_{1/2}$ in $x$ tangential to $\Sigma_{1/2}$ exist, and the following Taylor expansion holds for every curve $\lambda(t)\in \Sigma_{1/2}$ with $\lambda(0)=x$ and $\dot{\lambda}(0)=w\in T_x\Sigma_{1/2}$:
\begin{align}\label{eq:taylorphi}
 \phi_{1/2}(\lambda(t))=\phi_{1/2}(\lambda(0))+t\langle \nabla \phi_t(\lambda(0)),w\rangle+ \frac{t^{2}}{2}\langle A_{x}w,w\rangle+o(t^2).
\end{align}
Though \cite{rock} only considers the case of convex functions in $\mathbb{R}^n$, it is clear that the proof works as well in the context of Riemannian manifolds and semi-concave functions.
Finally, the Taylor expansion \eqref{eq:taylorphi} implies that $A_{x}$ must be self-adjoint.
\\In order to conclude the proof, observe that  $(0,1)\ni t\mapsto D T_{1/2}^t (x)$ is $C^{1}$ and that, by the symmetry of second order derivatives of distributions, for $\mu_{1/2}$-a.e. $x\in \Sigma_{1/2}$ it holds
$$D_t|_{t=1/2}(D T_{1/2}^t (x))=D(\nabla_t|_{t=1/2} T_{1/2}^t(x))=:Dv:=A_{x}.$$
\end{proof}

\begin{remark}\label{rem:defB}
Let $\{\mu_t\}_{t\in [0,1]}$ be a countably $\mathcal{H}^p$-rectifiable $W_2$-geodesic, $s\in (0,1)$, and $x\in \Sigma_s\setminus N$ where $N\subset \Sigma_s$ with $\mu_s(N)=0$
is given by Lemma \ref{lem:vDiff}. Since 
a countably $\cH^{p}$-rectifiable set has $p$-dimensional euclidean tangent spaces $\cH^{p}$-a.e., without loss of generality we can assume that for
every $x \in \Sigma_{s}\setminus N$ it holds $\dim T_{x}\Sigma_{s}={p}$.
Choose an orthonormal basis   $(e_1,\dots,e_p)$ of $T_x \Sigma_{s}$ and consider the vector fields $J_1,\dots,J_p:[0,1]\rightarrow T_{\gamma_{x}(t)}M$ along 
the geodesic $\gamma_{x}:[0,1]\rightarrow M$ defined by
\begin{equation*}
J_{i}(t):=\left(DT_{s}^{t}(x)\right) [e_{i}]=\left(D( \exp_{(\cdot)}(t v(\cdot))) (x) \right) [e_{i}], \quad \forall i = 1\ldots,p,  \quad \forall t \in [0,1]
\end{equation*}
where $v(x)$ was defined before in Lemma \ref{lem:vDiff}.
A standard computation of Riemannian geometry shows that the map $t\mapsto J_{i}(t)$ satisfies the Jacobi equation 
\begin{align*}
\nabla_t\nabla_t J_i+R(\dot{\gamma}_x, J_i)\dot{\gamma}_x=0,  \quad \forall  i=1,\dots,p, \quad \mbox{on }[0,1],
\end{align*}
where $\nabla_t$ is the covariant derivative of vector fields along $\gamma_x$ at the point $\gamma_{x}(t)$. In other words, $J_i$ is a Jacobi field. We then set 
\begin{equation}\label{eq:defBxt}
{B}_{x}(t): T_{x} \Sigma_{s} \to  T_{\gamma_{x}(t)}M, \quad {B}_{x}(t):=DT_{s}^{t}(x)
\quad \forall t\in [0,1], \; \forall x \in \Sigma_{s}\setminus N.
\end{equation} 
The combination of Lemma \ref{B} and Lemma \ref{lem:vDiff} yields that $B_x(t)$ is non-degenerate for every $t\in 
[0,1]$ for $\mu_s$-almost every $x\in \Sigma_s$. 
So in particular, for $\mu_s$-a.e. $x$ we have that $\dim[ \mbox{Im}[DT_{s}^t(x)]]=p$ and $\left\{J_i(t)\right\}_{i=1,\dots,p}$ is a basis of $\mbox{Im}[DT_s^t(x)]$ 
for 
every$t\in
[0,1]$. We can  (and will) consider $B_x(t)$ as a map from $T_x\Sigma_{s}$ to $T_{\gamma_{x}(t)}\Sigma_t$. Finally, we also proved that $D_{t} B_x(t)|_{t=s}:T_x\Sigma_{s} \to T_x\Sigma_{s}$ is self-adjoint for  $\mu_s$-almost every $x\in \Sigma_s$.
\end{remark}

\section{Jacobi fields computations}\label{sec:Jacobi}
\noindent
Let $(M,g)$ be a complete Riemannian manifold without boundary, and let $\gamma:[0,1]\rightarrow M$ be a minimizing, constant speed geodesic with $\gamma(0)=x$. Moreover, let $\{e_i\}_{i=1,\dots,p}$ be orthonormal vectors in $T_xM$, 
and let $J_{e_i}:[0,1]\rightarrow TM$ be non-vanishing Jacobi fields along $\gamma$ with $J_i(0)=e_i$
and  $J_i'(0)=f_i$, for some $f_{i}\in T_{x}M$ to be specified later. We denote with $T_{\gamma(t)}\Sigma_t\subset T_{\gamma(t)}M$ the span of $\left\{J_{e_i}(t)\right\}_{i=1,\dots,p}$ for each $t\in [0,1]$, and with 
$v^{\top}$ the orthogonal projection of a vector $v\in T_{\gamma(t)}M$ to the subspace $T_{\gamma(t)}\Sigma_t$. Similarly, $v^{\perp}$ is its projection to the orthogonal complement  $(T_{\gamma(t)}\Sigma_t)^{\perp}$ of $T_{\gamma(t)}\Sigma_t$. We also denote with $\top:T_{\gamma(t)}M \to T_{\gamma(t)}\Sigma_t$ 
 the orthogonal projection map. 
\begin{lemma}\label{lem:Ei}
Define the vector fields $E_i:[0,1]\rightarrow T_{\gamma(t)}\Sigma_t$, $i=1,\dots,p$, along $\gamma$ with values in $\bigcup_{t\in [0,1]}T_{\gamma(t)}\Sigma_t$
as the solution of 
\begin{equation}\label{eq:defEi}
(\nabla_t E_i)^{\top}=0, \quad \text{with  }E_i(0)=e_i,
\end{equation}
where $\nabla_t$ is the covariant derivative of vector fields along $\gamma$ at the point $\gamma(t)$.
Then $\left\{E_i(t)\right\}_{i=1,\dots,p}$ is an orthonormal basis for $T_{\gamma(t)}\Sigma_t$ for every $t\in [0,1]$.
\end{lemma}
\begin{proof}
The existence and uniqueness of $E_i:[0,1]\rightarrow T_{\gamma(t)}\Sigma_t$ solving \eqref{eq:defEi} is standard as it corresponds to solve a system of first order linear homogeneous ODEs with Cauchy conditions.
By definition of $E_i$, $i=1,\dots,p$ we have 
\begin{align*}
\frac{d}{dt}\langle E_i,E_j\rangle= \langle \nabla_t E_i,E_j\rangle + \langle E_i,\nabla_t E_j\rangle = \langle (\nabla_t E_i)^{\top},E_j\rangle + \langle E_i,(\nabla_t E_j)^{\top}\rangle = 0.
\end{align*}
Hence, $\langle E_i,E_j\rangle$ is constant along $\gamma$, and since $E_i(0)=e_i$, $i=1,\dots,p$, is an orthonormal basis of $T_{\gamma_0}\Sigma_0$ the claim follows.
\end{proof}
In the following we denote $D_t:=\top\circ\nabla_t$. For $E_i$ as in  Lemma \ref{lem:Ei},  by construction we have $D_tE_i=0$.

Let $B(t):T_{\gamma(0)}\Sigma_0\rightarrow T_{\gamma(t)}M$ be the $1$-parameter family of linear maps defined via $B(t)e_i=J_{e_i}(t)$, and consider  $\nabla_tB(t):T_{\gamma_0}\Sigma_0\rightarrow T_{\gamma(t)}M$ given by $(\nabla_tB(t))e_i=\nabla_tJ_{e_i}$. 
If we consider $B(t)$ as a map from $T_{\gamma(0)}\Sigma_0$ to $T_{\gamma(t)}\Sigma_t$, its derivative $D_t{B}(t)$ defined by $\left[D_t{B}(t)\right]e_i= D_tJ_{e_i}$ 
is a map from $T_{\gamma_0}\Sigma_0$ to $T_{\gamma(t)}\Sigma_t$ 
as well.
Moreover, since $\left\{J_{e_i}\right\}_{i=1,\dots,p}$ are Jacobi fields in $M$, the Jacobi equation yields
\begin{align}\label{eq:jacobi}
\nabla_t\nabla_tB(t) + R(\dot{\gamma}_t,B(t))\dot{\gamma}_t=0.
\end{align}

In the rest of the section  we are going to work under the assumption that  $B(t):T_{\gamma(0)}\Sigma_0 \to T_{\gamma(t)}\Sigma_t$ is invertible for all $t\in [0,1]$, in fact that will be satisfied in the optimal transport  application of the next section thanks to Lemma \ref{B}.  It will be  convenient to  consider the  operators:
\begin{align}
\mathcal{U}(t)&:= (\nabla_tB(t)){B}(t)^{-1}: T_{\gamma(t)}\Sigma_t\rightarrow T_{\gamma(t)}M  \label{eq:CalU}\\
{\mathcal{U}}^{\top}(t)&:=[\mathcal{U}(t)]^{\top}= (D_t{B}(t)){B}(t)^{-1}:T_{\gamma(t)}\Sigma_t\rightarrow T_{\gamma(t)}\Sigma_t. \label{eq:CalUTop} \\
{\mathcal{U}}^{\perp}(t)&:=[\mathcal{U}(t)]^{\perp}:T_{\gamma(t)}\Sigma_t\rightarrow (T_{\gamma(t)}\Sigma_t)^{\perp}. \label{eq:CalUperp} 
\end{align}

\begin{lemma}\label{prop:comp}
Let $J_i:=J_{e_{i}}$ and $E_i$, $i=1,\dots,p$ be as above.
Then
\begin{align*}
T_{\gamma(t)}\Sigma_t^{\perp}\ni \nabla_t E_i(t)= \mathcal{U}^{\perp}(t)E_i.
\end{align*}
\end{lemma}
\begin{proof}
First, we write $J_i=\sum_{j=1}^p\langle J_i,E_j\rangle E_j$ and set $A_{ij}=\langle J_i,E_j\rangle$ where the matrix $A:=(A_{ij})_{i,j}\in GL_n(\mathbb{R})$. Let $A^{-1}$  be its inverse.
We compute 
\begin{align*}
\nabla_tJ_i(t)=\sum_{j=1}^p \langle \nabla_t J_i(t),E_j(t)\rangle E_j(t) + \sum_{j=1}^p\langle J_i(t),\nabla _tE_j(t)\rangle E_j(t) + \sum_{j=1}^p\langle J_i(t),E_j(t)\rangle \nabla_tE_j(t),
\end{align*}
where the second sum on the right hand side vanishes since $(\nabla_tE_j(t))^{\top}=0$. Rearranging terms and multiplying by $A^{-1}$ yields for $k=1,\dots ,p$
\begin{align*}
\sum_{i=1}^p (A^{-1})_{ki} (\nabla_t J_i(t))^{\perp}&=\sum_{i=1}^p (A^{-1})_{ki}  \left[\nabla_tJ_i(t)-\sum_{j=1}^p \langle \nabla_t J_i(t),E_j(t)\rangle E_j(t)\right]  \\
&=  \sum_{i=1}^p (A^{-1})_{ki}  \left[ \sum_{j=1}^p \langle J_i(t),E_j(t)\rangle   \nabla_t E_j(t)   \right]  = \sum_{i=1}^p\sum_{j=1}^p (A^{-1})_{ki} A_{ij} \nabla_t E_j(t)= \nabla_t E_k(t).
\end{align*}
Now, we recall that $\nabla_t J_i= \nabla _tJ_{E_i}=\nabla _tJ_{B^{-1}(t)J_i}=\mathcal{U}(t)J_i$. Therefore 
\begin{align*}
\nabla_t E_k(t) =\sum_{i=1}^p (A^{-1})_{ki}\, (\nabla_t J_i(t))^{\perp} = \sum_{i=1}^p (A^{-1})_{ki}\, \left[ \mathcal{U}(t)J_i \right]^{\perp} =\left[ \mathcal{U}(t) \left(\sum_{i=1}^p   (A^{-1})_{ki}   J_i\right) \right]^{\perp} =[\mathcal{U}(t)E_k(t)]^{\perp},
\end{align*}
as desired.
\end{proof}

\begin{lemma}
Let ${B}(t):T_{\gamma(0)}\Sigma_0\rightarrow T_{\gamma(t)}\Sigma_t$, $t\in [0,1]$, be as above, and ${B}(t)^{-1}:T_{\gamma(t)}\Sigma_t\rightarrow T_{\gamma_0}\Sigma_0$. Then 
\begin{align}\label{eq:DtB}
D_t[{B}(t)^{-1}]= -{B}(t)^{-1} (D_t{B}(t)){B}(t)^{-1}.
\end{align}
\end{lemma}
\begin{proof}
Let $\left\{E_i\right\}_{i=1,\dots,p}$ be as in the previous lemma. Then, we obviously have ${B}(t){B}(t)^{-1}E_i(t)=E_i(t)$ for any $i=1,\dots,p$. Applying $D_t$ yields
\begin{align*}
(D_t{B}(t)) {B}(t)^{-1}E_i(t) + {B}(t) (D_t ({B}(t)^{-1})) E_i(t) + {B}(t){B}(t)^{-1}(D_tE_i(t))=D_tE_i(t)=0.
\end{align*}
Rearranging the terms and applying ${B}(t)^{-1}$ from the left of both sides yields the claim.
\end{proof}

The next proposition expresses the ``$p$-dimensional volume distortion'' along the geodesic $\gamma$ in terms of the $p$-Ricci curvature  and will be crucial for proving the characterization of lower curvature bounds in terms of optimal transport in the next section.   

\begin{proposition}\label{cor:imp}
Let ${\mathcal{U}}(t), {\mathcal{U}}(t)^{\top}, {\mathcal{U}}(t)^{\perp}$ be defined in \eqref{eq:CalU}, \eqref{eq:CalUTop}, \eqref{eq:CalUperp}. Then it holds
\begin{align*}
\nabla_t {\mathcal{U}}(t) + \mathcal{U}(t){\mathcal{U}}^{\top}(t) + R(\dot{\gamma}_t,\cdot)\dot{\gamma}_t=0.
\end{align*}
Taking the trace along  $T_{\gamma(t)}\Sigma_t$ yields
$$
\tr(D_t {\mathcal{U}}(t)) + \tr(({\mathcal{U}}^{\top}(t))^2) + \ric_p(T_{\gamma(t)}\Sigma_t,\dot{\gamma})=0,
$$
and moreover
\begin{align}\label{eq:truT'}
\tr(\mathcal{U}^{\top}(t))'+ \tr(({\mathcal{U}}^{\top}(t))^2) + \ric_p(T_{\gamma(t)}\Sigma_t,\dot{\gamma}(t))= \|\mathcal{U}^{\perp} (t)\|^{2}.   
\end{align}
If $D_{t}B(t)|_{t=0}:T_{\gamma(0)}\Sigma_0 \to T_{\gamma(0)}\Sigma_0$ is self-adjoint then  ${\mathcal{U}}^{\top}(t): T_{\gamma(t)}\Sigma_t \to T_{\gamma(t)}\Sigma_t$ is self-adjoint for all $t \in [0,1]$ and, setting $y(t)=\log\det B(t)$, it holds
\begin{equation}\label{eq:eqdiffyx}
y''(t)+ \frac{1}{p}y'(t)^2  +  \ric_p(T_{\gamma(t)}\Sigma_t,\dot{\gamma}(t)) - \| \mathcal{U}^{\perp}(t)  \|^{2} \leq 0.
\end{equation}
\end{proposition}

\begin{remark}
In case $p={\rm dim}(M)$ then ${\mathcal U}^{\top}(t)={\mathcal U}(t)$,  ${\mathcal U}^{\perp}(t)=0
$ and  $\ric_p(T_{\gamma(t)}\Sigma_t,\dot{\gamma}(t))=\ric(\dot{\gamma}(t), \dot{\gamma}(t))$, so that Proposition 
\ref{cor:imp} recovers the classical Jacobian estimates expressing the volume distortion along a geodesic in terms of Ricci curvature (see for instance \cite[Lemma 3.1]{CMS2}).
\end{remark}

\begin{proof}
First of all, there is a natural extension of $B(t)$ (and of $\nabla_{t}B(t)$) 
to maps from the whole $T_{\gamma(0)}M$ just by composing with the orthogonal projection into $T_{\gamma(0)}\Sigma_{0}$, i.e. for $v \in T_{\gamma(0)}M$ we consider $B(t) v^{\top}$.
Differentiating the identity $\top \circ \top =\top$ gives $\nabla_{t}\top \circ \top +\top \circ\nabla_{t} \top =\nabla_{t} \top$; left and right composing with  $\top$, yields $\top \circ \nabla_{t}\top \circ \top =0$. 
Therefore, using \eqref{eq:jacobi} and \eqref{eq:DtB},  we get
\begin{align*}
\nabla_t \left[\mathcal{U}(t)\right] &= \left[\nabla_t\nabla_tB(t)\right]B(t)^{-1} +  \nabla_tB(t) (\top \circ \nabla_{t}\top \circ \top)  B(t)^{-1} +  \nabla_tB(t)  [\nabla_{t}B(t)^{-1}]^{\top} \\
&= - R(\dot{\gamma}(t),\cdot)\dot{\gamma}(t)  +  \nabla_tB(t)  [D_{t}B(t)^{-1}] = - R(\dot{\gamma}(t),\cdot)\dot{\gamma}(t)  - \nabla_tB(t)B(t)^{-1}D_tB(t)B(t)^{-1} \\
&= - R(\dot{\gamma}(t),\cdot)\dot{\gamma}(t)   - \mathcal{U}(t)\mathcal{U}^{\top}(t).
\end{align*}
Taking the trace along $T_{\gamma(t)}\Sigma_t$ yields the second identity. 
To get the  identity \eqref{eq:truT'},  observe that
$\tr {\mathcal{U}^{\top}}(t)=\sum_{i=1}^p\langle \mathcal{U}(t)E_i(t),E_i(t)\rangle$ and 
\begin{align*}
\langle \mathcal{U}(t)E_i(t),E_i(t)\rangle'&= \left\langle \left[D_t \mathcal{U}(t)\right] E_i(t),E_i(t)\right\rangle +\left\langle \mathcal{U}(t)\left[D_tE_i(t)\right],E_i(t)\right\rangle + \left\langle \mathcal{U}(t)E_i(t),\nabla_t E_i(t)\right\rangle.
\end{align*}
Since $D_tE_i=(\nabla_t E_i)^{\top}=0$ and $\nabla_tE_i=(\nabla_tE_i)^{\perp}$, we conclude that
\begin{align*}
(\tr{\mathcal{U}^{\top}}(t))'= \tr(D_t\mathcal{U}(t)) + \sum_{i=1}^p\langle \mathcal{U}(t) E_{i}(t), \nabla_t E_i(t)\rangle =\tr(D_t{\mathcal{U}}(t)) + \sum_{i=1}^p\langle (\mathcal{U}(t)E_i(t))^{\perp}, \nabla_t E_i(t)\rangle.
\end{align*}
In particular let us explicitly observe that, in general, $\tr(D_t{\mathcal{U}}(t))\neq \tr({\mathcal{U}^{\top}}(t))'$. The claimed identity \eqref{eq:truT'} follows by observing that   
$\|\mathcal{U}(t)^{\perp}\|^{2}= \sum_{i=1}^p\langle (\mathcal{U}(t)E_i(t))^{\perp}, \nabla_t E_i(t)\rangle$.

The rest of the proof is devoted to show \eqref{eq:eqdiffyx}. Setting $y(t)=\log\det B(t)$, we have that
\begin{align}
y'(t_{0}) &= \left. \frac{d}{dt} \right|_{ t=t_{0}} \log\det \big( B(t)B(t_{0})^{-1} \big)= \left. \frac{d}{dt} \right|_{ t=t_{0}} \log\det \left[ \big(  \langle B(t)B(t_{0})^{-1} E_{i}(t), E_{j}(t) \rangle \big)_{i,j} \right]  \nonumber \\
&= \tr \left[(D_{t} B(t)) B(t_{0})^{-1} \right]|_{t=t_{0}}+  2 \sum_{i=1}^{p} \langle D_{t} E_{i}(t), E_{i}(t) \rangle |_{t=t_{0}} \nonumber \\
& =  \tr \left[(D_{t} B(t)) B(t_{0})^{-1} \right]|_{t=t_{0}} = \tr(\mathcal U^{\top}(t_{0})), \quad  \label{eq:y'}
\end{align}
since by construction $D_{t} E_{i}(t)=0$.

We next claim that, under the assumption that  $D_{t}B(t)|_{t=0}$ is self-adjoint, then
\begin{equation}\label{eq:claimUTSA}
\mathcal U^{\top}(t):T_{\gamma(t)} \Sigma_{t} \to T_{\gamma(t)} \Sigma_{t} \quad   \text{ is self-adjoint for all $t \in [0,1]$.}
 \end{equation}
 To this aim, calling $(\mathcal U^{\top}(t))^{*}$ the adjoint operator, we observe that 
\begin{equation}\label{eq:U-U*}
(\mathcal U^{\top}(t))^{*}-\mathcal U^{\top}(t)=(B(t)^{*})^{-1} \left[ (D_{t}B(t)^{*}) B(t)   - B(t)^{*} (D_{t}B(t)) \right]   B(t)^{-1},
\end{equation}
and that 
\begin{equation}\label{eq:DtUU*}
D_{t} \left[ (D_{t}B(t)^{*}) B(t)   - B(t)^{*} (D_{t}B(t)) \right] = (D^{2}_{t}B(t)^{*}) B(t)   - B(t)^{*} (D^{2}_{t}B(t)).
\end{equation}
Now, combining the Jacobi equation \eqref{eq:jacobi} with the identity $\top \circ \nabla_{t}\top \circ \top =0$ proved at the beginning of the proof, we have
\begin{align}
D_{t}^{2} B(t)&=\top \nabla_{t}\left( \top \nabla_{t} B(t)  \right)= \top  \left(\nabla_{t}^{2} B(t)  \right)+ \top (\nabla_{t}  \top) \top \nabla_{t} B(t)  =  \top  \left(\nabla_{t}^{2} B(t)  \right) \nonumber \\
&= -  \left( R(\dot{\gamma}(t),B(t))\dot{\gamma}(t) \right)^{\top} =- {\mathcal R}(t) B(t),  \label{eq:D2B}
\end{align}
where 
\begin{equation}
{\mathcal R}(t):T_{\gamma(t)} \Sigma_{t} \to T_{\gamma(t)} \Sigma_{t}, \quad  {\mathcal R}(t)[v]:=\left[R(\dot{\gamma}(t), v) \dot{\gamma}(t)  \right]^{\top}
\end{equation}
is self-adjoint; indeed, in the orthonormal basis $\{E_{i}(t)\}_{i=1,\ldots,p}$, it  is represented by the symmetric matrix $\langle R(\dot{\gamma}(t), E_{i}(t)) \dot{\gamma}(t), E_{j}(t)\rangle$. Plugging \eqref{eq:D2B} into \eqref{eq:DtUU*}, we obtain that $(D_{t}B(t)^{*}) B(t)   - B(t)^{*} (D_{t}B(t))$ is constant in $t$ and thus vanishes identically, since by assumption $B(0)={\rm Id}$ and $D_{t}B(t)|_{t=0}$ is self-adjoint.
Taking into account \eqref{eq:U-U*}, this concludes the proof of the claim \eqref{eq:claimUTSA}.
\\Using that ${\mathcal U}^{\top} (t)$ is a  self-adjoint operator over a $p$-dimensional space, by Cauchy-Schwartz inequality,  we have that  
\begin{equation}\label{eq:CSUT}
\tr\big[({\mathcal U}^{\top} (t))^{2} \big]\geq \frac{1}{p} \left(\tr\big[\mathcal U^{\top} (t) \big] \right)^{2}.
\end{equation}
 The desired estimate \eqref{eq:eqdiffyx} then follows  from the combination of \eqref{eq:truT'}, \eqref{eq:y'} and \eqref{eq:CSUT}.
\end{proof}

\noindent

In the final part of the section we specialize to the case $p=1$, giving the self-contained easier arguments.

\begin{proposition}
Assume $p=1$,  let $J:=J_{e_1}$ and $E:=E_1$ be as above. In particular, $\dim T_{\gamma(t)}\Sigma_t = 1$ for every $t\in [0,1]$, and $E=|J(t)|^{-1}J(t)$. Then 
\begin{align}\label{eq:NEt}
\nabla_t E(t) = |J(t)|^{-1} (\nabla_t J(t))^{\perp}.
\end{align}
\end{proposition}
\begin{proof}
We compute $\nabla_t E$ as follows
\begin{align*}
\nabla_tE(t)=\left(|J(t)|^{-1}\right)' J(t) + |J(t)|^{-1} \nabla_t J(t).
\end{align*}
Since
\begin{align*}
\left(|J|^{-1}\right)'= \left(\langle J,J\rangle^{-\frac{1}{2}}\right)'= - |J|^{-3}\langle J,\nabla_{t}J\rangle= -|J|^{-2}\langle E, \nabla_{t} J\rangle,
\end{align*}
we get
\begin{align*}
\nabla_tE= -|J|^{-1} \langle E,\nabla_t J\rangle E + |J|^{-1} \nabla_t J= |J|^{-1} (\nabla_t J)^{\perp}.
\end{align*}
\end{proof}
\begin{corollary}\label{cor:modric}
Assume $p=1$, and consider ${\mathcal{U}^{\top}}(t)$ as above. Then, we have
\begin{align*}
\langle \mathcal{U}^{\top}(t)E(t),E(t)\rangle' + \langle \mathcal{U}^{\top}(t)E(t),E(t)\rangle^2 + \ric_1(T_{\gamma(t)}\Sigma_t, \dot{\gamma}_t)= |(\mathcal{U}(t)E(t))^{\perp}|^2.
\end{align*}
\end{corollary}
\begin{proof} Since 
\begin{align*}
\tr((\mathcal{U}^{\top})^2)=\langle ({\mathcal{U}}^{\top})^2(t)E(t),E(t)\rangle=\langle {\mathcal{U}^{\top}}(t)E(t),E(t)\rangle^2,
\end{align*}
and, from \eqref{eq:NEt}, we have
\begin{align}\label{eq:UtperpJe}
\nabla_{t}E(t) = |J(t)|^{-1} (\nabla_{t}J(t))^{\perp} = |J(t)|^{-1} (\nabla_{t}J_{B(t)^{-1}J(t)})^{\perp}= |J(t)|^{-1} (\mathcal{U}(t)J(t))^{\perp}= (\mathcal{U}(t)E(t))^{\perp},
\end{align}
the claim  follows from \eqref{eq:truT'}.
\end{proof}

\section{OT characterization of  sectional curvature  upper bounds}\label{sec:UB}
\begin{proposition}\label{prop:kconvex}
Let $M$ be a complete Riemannian manifold without boundary with $\ric_1\leq K$,  let $\{\mu_t\}_{t\in [0,1]}$ be a countably $\mathcal{H}^1$-rectifiable $W_2$-geodesic, and 
consider ${B}_x(t):T_{x}\Sigma_0\rightarrow T_{\gamma(t)}\Sigma_t$, $t\in [0,1]$ as in Remark \ref{rem:defB}. Then, the function
$[0,1]\ni t\mapsto  \mathcal{J}_x(t):=\det{B}_x(t)\in \R$ belongs to $C([0,1])\cap C^{2}((0,1))$ and satisfies 
\begin{align}\label{inequ:kconvex}
  \mathcal{J}_x'' + K|\dot{\gamma}|^2\mathcal{J}_x\geq 0\ \ \mbox{ on } (0,1).
\end{align}
In particular,
if $t_0,t_1\in [0,1]$, $\tau(s)=s(t_1-t_0)+t_0$, and $s\in[0,1]\mapsto\varsigma_s=\gamma_{\tau(s)}$, 
we have for all $s\in [0,1]$
\begin{align}\label{inequ:integrated}
\cJ_x(\tau(s))\leq \sigma_{K,1}^{(1-s)}(|\dot{\varsigma}|)\, \cJ_x(t_0) + \sigma_{K,1}^{(s)}(|\dot{\varsigma}|)\, \cJ_x(t_1).
\end{align}
\end{proposition}
\begin{proof}
First note that, setting $t\in[0,1]\mapsto y_x(t):= \log\mathcal{J}_x(t)$ we have  that $y_x'(t)= \tr((D_t{B}_x(t)){B}_x^{-1}(t))=\tr \mathcal{U}_x(t)=\tr \mathcal{U}_x(t)^{\top}$. Then Corollary \ref{cor:modric} yields that 
\begin{align*}
y''_x+(y_x')^2+K|\dot{\gamma}_x|^2\geq 0 \ \ \mbox{ on }  (0,1).
\end{align*}
By computing $\mathcal{J}_x''(t)=\left(e^{y_x}\right)''(t)$ this yields (\ref{inequ:kconvex}).
Moreover, considering $t_0,t_1,\tau$ and $\varsigma$ as above we get
\begin{align*}
\frac{d^2}{ds^2}\cJ_x\circ\tau + K|\dot{\varsigma}|^2 \cJ_{x}\circ\tau \geq 0 \  \ \mbox{ on }  (0,1).
\end{align*}
that is equivalent to (\ref{inequ:integrated}) by classical comparison principle.
\end{proof}
\begin{theorem}[Curvature upper bounds]\label{thm:USB}
Let $(M,g)$ be a complete Riemannian manifold without boundary and let $K\geq 0$. Then the following statements {\rm (i)} and {\rm (ii)} are equivalent: 
\begin{itemize}
\item[(i)] $\ric_{1}\leq K$ or, equivalently, $\sect\leq K$.
\medskip
\item[(ii)]
Let $\{\mu_t\}_{t\in [0,1]}$ be a countably $\mathcal{H}^1$-rectifiable $W_2$-geodesic, and let $\Pi$ be the corresponding dynamical transport plan. Then, if $t_0,t_1\in (0,1)$ and  $\tau(s)=(1-s)t_0+st_1$, it holds
\begin{align*}
\cH^{1}(\supp \mu_{\tau(s)})\leq \int\left[\sigma_{K,1}^{(1-s)}(|\dot{\gamma\circ\tau}|)\rho_{t_0}(\gamma(t_0))^{-1}+\sigma_{K,1}^{(s)}(|\dot{\gamma\circ \tau}|)\rho_{t_1}(\gamma(t_1))^{-1}\right]d\Pi(\gamma), \ \ \forall s \in [0,1],
\end{align*}
where $\rho_t$ is the density of $\mu_t$ w.r.t. $\mathcal{H}^1$.
\end{itemize}
In the case of $K=0$ the inequality in {\rm (ii)} becomes
\begin{align*}
\cH^{1}(\supp \mu_{\tau(s)})\leq (1-s)\cH^{1}(\supp\mu_{t_0}) + s\cH^{1}(\supp\mu_{t_1}), \ \ \forall s \in [0,1].
\end{align*}
\end{theorem}

\begin{remark} \label{rem:UBK>0}
Recall  from Remark \ref{rem:p12nn-1} that the condition $\ric_{1}\leq K<0$ is never satisfied as $\ric_{1}(\R v,v)=0$ for every $v\in TM$; hence  it makes sense just  to assume a non-negative upper bound $K\geq 0$ and, in this case, $\ric_{1}\leq K$ is equivalent to $\sect \leq K$.
\end{remark}
\begin{remark}\label{rem:SharpUB}
In the assertion (ii) of Theorem \ref{thm:USB},  one cannot relax the assumption to  $t_0,t_1\in [0,1]$. For instance, one can consider a 
cylinder $\mathbb{R}\times \mathbb{S}^1$ that is a space of zero (in particular non-positive) sectional curvature.  Parametrize $\mathbb{S}^1$  by arclength on $[0,2\pi]$, in particular $0$ and $\pi$ are two antipodal points in $\mathbb{S}^1$. 
Then, the uniform distribution on the set of all geodesics connecting $(s,0)$ and $(s,\pi)$ for $s\in [0,1]$ defines 
a countably $\mathcal{H}^1$-rectifiable $W_2$-geodesic $\left\{\mu_t\right\}_{t\in [0,1]}$ such that $\supp\mu_{0}=[0,1]\times \left\{0\right\}, \, \supp\mu_{1}=[0,1]\times \left\{\pi\right\}, \,   \supp\mu_{1/2}=[0,1]\times \left\{\pi/2\right\} \cup [0,1]\times \left\{3\pi/2\right\}$. Hence, we have $\cH^1(\supp\mu_{1/2})=2$, $\cH^1(\supp\mu_0)=\cH^1(\supp\mu_1)=1$.
\end{remark}
\begin{proof}
(i) $\Longrightarrow$ (ii).  Let $\{\mu_{t}\}_{t\in [0,1]}$ be a countably $\cH^1$-rectifiable $W_2$-geodesic, i.e. for every $t\in [0,1]$ the probability  measure $\mu_{t}$ is concentrated on a countably  $\cH^{1}$-rectifiable set 
$\Sigma_{t}\subset M$ and is $\cH^{1}|_{\Sigma_t}$-absolutely continuous.  Also, thanks to Theorem \ref{thm:MongeMather} (see also Remark \ref{rem:MM}),  the $W_{2}$-geodesic $\{\mu_t\}_{t \in [0,1]}$ is given by Lipschitz optimal transport maps; more precisely there exist 
unique Lipschitz maps $T_{1/2}^{t}:\Sigma_{1/2} \to \Sigma_{t}$ such that $(T_{1/2}^{t})_{\sharp}{\mu_{1/2}}=\mu_{t}$ and $\pi_{0,1}:=(T_{1/2}^{0}, T_{1/2}^{1})_{\sharp} \mu_{1/2}$ is an
optimal coupling between $\mu_{0}$  and $\mu_{1}$. 
Let $\gamma_{x}(t):=T_{1/2}^t(x)\in \Geo(M)$ and $\gamma_x\circ\tau(s)=:\varsigma_x(s)$.

The map $x\mapsto (t\mapsto T_{1/2}^t(x))=:\gamma_x\in \Geo(X)$ yields a measurable map from $M$ into the space of geodesics $\Geo(X)$, and the push-forward of $\mu_{1/2}$ under 
this map is the associated optimal dynamical transport plan $\Pi$. In particular
\begin{align}\label{eq:trans}\int f(\gamma)\, d\Pi(\gamma)=\int f(\gamma_{x})\, d\mu_{1/2}(x)
\end{align}
for any non-negative measurable function $f:\Geo(X)\rightarrow [0,\infty]$.

Setting  $\cJ_{x}(t):=\det[DT_{1/2}^{t}(x)]$ for $\mu_{1/2}$-a.e. $x\in \Sigma_{1/2}$ and   making use of Lemma \ref{B} and Proposition \ref{prop:kconvex},  we can compute for every $ s \in [0,1]$: 
 \begin{align*}
\int_{\Sigma_{\tau(s)}}\rho_{\tau(s)}(x)^{-1}d\mu_{\tau(s)}(x)&= \int_{\Sigma_{1/2}} \rho_{\tau(s)}(T_{1/2}^{\tau(s)}(y))^{-1} \, d\mu_{1/2}(y)\\
& \overset{\eqref{eq:MAineq}}{=}
\int_{\Sigma_{1/2}} \cJ_x(\tau(s))d\mathcal{H}^1(x)\\
&\overset{ \eqref{inequ:integrated}}{\leq} 
\int_{\Sigma_{1/2}}\Big[\sigma_{K,1}^{(1-s)}(|\dot{\varsigma}_x|) \, \cJ_x(t_0)+\sigma_{K,1}^{(s)}(|\dot{\varsigma}_x|) \, \cJ_x(t_1)\Big]d\mathcal{H}^1(x)\\
&\overset{\eqref{eq:MAineq}}{=} 
\int_{\Sigma_{1/2}}\Big[\sigma_{K,1}^{(1-s)}(|\dot{\varsigma}_x|)\rho_{t_0}(T_{1/2}^{t_0}(x))^{-1}+\sigma_{K,1}^{(s)}(|\dot{\varsigma}_x|)\rho_{t_1}(T_{1/2}^{t_1}(x))^{-1}\Big]\rho_{1/2}(x) \, d\mathcal{H}^1(x)\\
&\overset{(\ref{eq:trans})}{=} \int\left[\sigma_{K,1}^{(1-s)}(|\dot{\gamma\circ\tau}|)\, \rho_{t_{0}}(\gamma(t_0))^{-1}+\sigma_{K,1}^{(s)}(|\dot{\gamma\circ\tau}|)\, \rho_{t_{1}}(\gamma(t_1))^{-1}\right]d\Pi(\gamma).
\end{align*}
Note that the assumption $t_{0},t_{1}\in (0,1)$ was used above in order to apply \eqref{eq:MAineq} with equality.

(ii) $\Longrightarrow$ (i). 
\\We argue by contradiction. Assume there exist $x_{0}\in M$, 
a line $P\subset T_{x_{0}}M$ and $0\neq v\in T_{x_{0}}M$ such that 
the $1$-Ricci curvature of $P$ in the direction of $v$ satisfies 
\begin{equation}
\ric_1(P,v)>(K+3\epsilon)|v|^2, 
\end{equation}
 for some $\epsilon>0$.
Let $\delta>0$ be sufficiently small such that $\exp_{x_{0}}|_{B_{\delta}(0)}$ is a diffeomorphism onto its image. Then
$\exp_{x_{0}}(P\cap B_{\delta}(0))=:\Sigma_{\scriptscriptstyle{\frac{1}{2}}}$ is a smooth $1$-dimensional submanifold. 
Let
$\phi\in C^{\infty}_0(M)$ be a Kantorovich potential such that 
\begin{equation}
\nabla \phi(x_{0}) = v \neq 0 \quad \text{ and }  \quad \nabla^2\phi(x_{0})=0. 
\end{equation}
By replacing $\phi$ with $\eta\phi$ for a sufficiently small number $\eta>0$ we get that $\phi$ is a Kantorovich potential as well and $|\nabla \phi|(y)$ 
is  smaller than the injectivity radius at $y$, for every $y\in \supp(\phi)\subset M$. It is easily checked that for $\delta>0$ small enough the map
$y\mapsto T_t(y)=\exp_y(-t\nabla \phi(y))$ is 
a diffeomorphism from $B_{\delta}(0)$ onto its image for any $t\in [-\frac{1}{2},\frac{1}{2}]$. Hence, $\Sigma_t:=T_{t-\frac{1}{2}}(\Sigma_{\frac{1}{2}})$ for $t\in [0,1]$ is a 1-parameter family 
of smooth $1$-dimensional submanifolds with finite $1$-dimensional Hausdorff measure. We define $\mu_{\frac 1 2 }:= \mathcal{H}^1(\Sigma_{\frac 1 2 })^{-1} \,\mathcal{H}^1 \llcorner \Sigma_{\frac 1 2}$;
note that  $\mu_{t}:=(T_{t-\frac{1}{2}})_{\sharp}\mu_{\frac  1 2}$, with $t \in [0,1]$, is the unique $L^2$-Wasserstein geodesic between $\mu_0$ and $\mu_1$. 
Moreover, by construction, $\mu_t$ is a $\mathcal{H}^1$-absolutely continuous probability measure concentrated on $\Sigma_{t}$.
\\

Calling $\gamma_{x}(t):=T_{t-\frac{1}{2}}(x)=\exp_{x}\big(-\big(t-\frac{1}{2}\big) \nabla \phi(x)\big)$ for $x \in \Sigma_{\frac 1 2}$ the geodesic performing the transport,  note  that by continuity  there exist  $\delta, \sigma>0$ small enough such that   
\begin{equation}\label{eq:AssAbsRicp0}
\ric_1(T_{\gamma_{x}(t)}\Sigma_{t}, \dot{\gamma}_{x}(t) )>  (K+2 \epsilon) |\dot{\gamma}_{x}(t)|^{2}, \quad \forall x \in \Sigma_{\frac 1 2}\subset B_{\delta}(x_{0}), \quad \forall  t \in \left[\frac{1}{2}-\sigma, \frac{1}{2}+\sigma \right] . 
\end{equation}
For every $x \in \Sigma_{\frac{1}{2}}$ note that  $\gamma_{x}(t):=T_{t-\frac{1}{2}}(x)$  is a  geodesic connecting  $T_{-\frac{1} {2}}(x)\in \Sigma_{0}$ to $T_{\frac{1}{2}}(x)\in \Sigma_{1}$.
Choose $e\in T_x\Sigma_{\frac{1}{2}}$,  consider the Jacobi field $J:[0,1]\rightarrow T_{\gamma_{x}(t)}M$ such that $J(\frac{1}{2})=e$ and $J'(\frac{1}{2})=[\nabla^2\phi(x)] e$, and set $|J(t)|^{-1}J(t)=E(t)$.  
We introduce again the linear operator ${\mathcal{U}}^{\top}(t)=D_tB_x(t)B_x(1)^{-1}$ where  $B_x(t):T_x\Sigma_{\frac 1 2}\rightarrow T_{\gamma_x(t)}M$ by $B_x(t)e:=J(t)=DT_{t-\frac{1}{2}}(x)e$. 
Then, as in Corollary \ref{cor:modric}, we get
\begin{equation}\label{eq:Jacobi0}
\langle {\mathcal{U}}^{\top}_x(t)E(t),E(t)\rangle' + \langle {\mathcal{U}}^{\top}_x(t)E(t),E(t)\rangle^2 + \ric_1(T_{\gamma(t)}\Sigma_t,\dot{\gamma}_t)= |(\mathcal{U}_x(t)E(t))^{\perp}|^2\leq |\cU_x(t)E(t)|^2.
\end{equation}
Since by construction $\mathcal{U}_{x_{0}}(0):=\nabla_{t}B_{x_{0}}(0)B_{x_{0}}^{-1}(0)= \nabla^2\phi (x_{0})|_{T_{x_{0}}\Sigma_{\frac 1 2}}=0$ and $v \neq 0$,  again by continuity we can choose $\delta,\sigma>0$ even smaller so that
\begin{equation}\label{eq:claim0}
|\mathcal{U}_x(t)E(t)|^2 < \epsilon |\dot{\gamma}_{x}(t)|^{2}, \quad \forall x \in \Sigma_{\frac 1 2} \subset B_{\delta}(x_{0}), \quad \forall t \in \left[  \frac 1 2 -\sigma , \frac 1 2 +\sigma  \right].
\end{equation}
The combination of  \eqref{eq:AssAbsRicp0}, \eqref{eq:Jacobi0} and   \eqref{eq:claim0} then yields
\begin{align*} 0 > \langle \cU^{\top}(t)E(t),E(t)\rangle' + \langle \cU^{\top}(t)E(t),E(t)\rangle^2+\left(K+  \epsilon  \right)|\dot{\gamma}_x(t)|^2, \; \forall x \in \Sigma_{\frac 1 2}, \; \forall t \in \left[  \frac 1 2 -\sigma , \frac 1 2 +\sigma  \right].
\end{align*}
Observe that the affine reparametrization $t=g(s)=  \frac 1 2 - \sigma + 2\sigma s$, $g:[0,1]\to \left[ \frac 1 2 - \sigma , \frac 1 2+ \sigma  \right]$, corresponds to consider the rescaled Kantorovich potential $2\sigma \phi$ in place of $\phi$ in the arguments above, and  thus gives  
\begin{align*}
0&> \langle \cU^{\top}(g(s))E(g(s)),E(g(s))\rangle' + \langle \cU^{\top}(g(s))E(g(s)),E(g(s))\rangle^2+\left(K + \epsilon  \right)|\dot{\gamma}_x(g(s))|^2\\
&\hspace{10cm}\forall x \in \Sigma_{\frac 1 2}, \; \forall s \in [0,1]. 
\end{align*}
Since $g$ is affine, the restricted and rescaled curve $\{\tilde{\mu}_{s}:=\mu_{g(s)}\}_{s \in [0,1]}$ is still a $W_{2}$-geodesic from $\tilde{\mu}_{0}=\mu_{\frac 1 2 -\sigma}$ to  $\tilde{\mu}_{1}=\mu_{1/2 +\sigma}$.
By repeating the arguments in the proof of (i)$\implies$(ii), with reversed inequalities and $K$ replaced by $K+\epsilon $,  we obtain
\begin{equation}
\int \Big[ \sigma_{K+ {\epsilon},1}^{\left(\frac{1}{2}\right)}(|\dot{\gamma\circ g}|) \tilde{\rho}_0(\gamma(1/2 -\sigma))^{-1} +\sigma_{K+ {\epsilon},1}^{\left(\frac{1}{2}\right)} (|\dot{\gamma\circ g}|) \tilde{\rho}_1(\gamma(1/2 +\sigma))^{-1} \Big] d\tilde{\Pi}(\gamma)<  \int_{\Sigma_{\frac 1 2}} \tilde{\rho}_{\frac 1 2}(y)^{-1} d \tilde{\mu}_{\frac 1 2}(y),
\end{equation}
where $\tilde{\Pi}$ is the optimal plan induced by the $W_{2}$-geodesic $\{\tilde{\mu}_s\}_{s \in [0,1]}$. 
Using that the distortion coefficients $\sigma_{K, 1}^{(t)}(\theta)$ are monotone increasing in $K$, we arrive to  contradict (ii) with $t=\frac 1 2$.
\end{proof}

We remind the reader that there is a notion of upper curvature bounds for geodesic metric spaces $(X,\sfd)$ that goes under the name ${\rm CAT}(K)$ for $K\in \mathbb{R}$ (see for instance \cite[Chapter 9]{bbi}). In case $K=0$, the condition reduces to require $1$-convexity of $\frac{1}{2}\sfd(y,\cdot)^2$ 
for any $y\in X$. For Riemannian
manifolds $(M,g_M)$ the condition ${\rm CAT}(K)$ for the induced metric space $(M,\sfd_M)$ implies an upper sectional curvature bound by $K$, moreover it also implies that geodesics are always extendible in case $K\leq 0$.  The next corollary then follows. 

\begin{corollary}
Let $(M,g)$ be a complete, simply connected Riemannian manifold without boundary. Then the following statements {\rm (i)} and {\rm (ii)} are equivalent: 
\begin{itemize}
\item[(i)] $(M,d_M)$ satisfies ${\rm CAT}(0)$.
\medskip
\item[(ii)]
Let $\{\mu_t\}_{t\in [0,1]}$ be a countably $\mathcal{H}^1$-rectifiable $W_2$-geodesic. Then
\begin{align*}
\cH^{1}(\supp \mu_{t})\leq (1-t)\cH^1(\supp\mu_0)+t\cH^1(\supp\mu_1),
\ \ \forall t \in [0,1].
\end{align*}
\end{itemize}
\end{corollary}
\begin{proof}
The implication (i)$\Rightarrow$(ii) follows from the extendibility of geodesics.
\\
The reverse implication follows from the reverse implication in Theorem \ref{thm:USB}. Indeed, the theorem implies that $M$ has non-positive sectional curvature, therefore the ${\rm CAT}(0)$-condition holds locally. Then, since $M$ is simply connected
the condition globalizes by  \cite[Theorem 9.2.9]{bbi}.
\end{proof}

\section{OT characterization of sectional, and more generally $p$-Ricci, curvature lower bounds}\label{sec:LB} 
Throughout the section,  $\{\mu_t\}_{t\in [0,1]}$ is a countably $\mathcal{H}^p$-rectifiable $W_2$-geodesic and  $\Pi$ is the corresponding dynamical optimal plan, i.e. 
$\mu_t=\rho_{t} \, \mathcal{H}^{p} \llcorner \Sigma_{t}$ where $\Sigma_{t}\subset M$ is a countably $\mathcal{H}^{p}$-rectifiable subset and $\rho_{t}\in L^{1}(M,\mathcal{H}^{p})$. 
From Theorem \ref{thm:MongeMather} (see also Remark \ref{rem:MM}),    we know that $\mu_{t}=(T^{t}_{1/2})_{\sharp} \mu_{1/2}$ with $T^{t}_{1/2}:\Sigma_{1/2}\to \Sigma_{t}$ is Lipschitz.  Lemma \ref{lem:vDiff} (see also Remark \ref{rem:defB})
ensures the existence of a  subset $N\subset \Sigma_{1/2}$, with $\mathcal{H}^{p}(N)=0$, such that $T^{t}_{1/2}$ is differentiable for every $x\in \Sigma_{1/2}\setminus N$ and we set
\begin{equation*}\label{eq:defBxt}
{B}_{x}(t): T_{x} \Sigma_{1/2} \to  T_{\gamma_{x}(t)} \Sigma_{t}, \quad {B}_{x}(t):=DT_{1/2}^{t}(x)
\quad \forall t\in [0,1], \; \forall x \in \Sigma_{1/2}\setminus N.
\end{equation*} 
Moreover $D_{t} B_x(t)|_{t=1/2}:T_x\Sigma_{1/2} \to T_x\Sigma_{1/2}$ is self-adjoint for every $x\in \Sigma_s\setminus N$.
\\Lemma \ref{B}  yields that $B_x(t)$ is invertible for every $t\in [0,1]$ for every $x\in \Sigma_{1/2}\setminus N$, up to enlarging the subset $N$.   Since $B_x(1/2)={\rm Id}$, it follows in particular that $\det[B_{x}(t)]>0$ for all $t \in [0,1]$. Now, for every $x\in \Sigma_{1/2}\setminus N$ and $t \in [0,1]$,  let $\gamma_{x}(t):=T^{t}_{1/2}(x)$ be the geodesic performing the transport  and consider 
\begin{align*}
\mathcal{U}_{x}(t)&:= (\nabla_tB_{x}(t)){B_{x}}(t)^{-1}: T_{\gamma_{x}(t)}\Sigma_t\rightarrow T_{\gamma_{x}(t)}M, \\
\mathcal{U}_{x}^{\top}(t)&:=[\mathcal{U}_{x}(t)]^{\top}= (D_t{B}_{x}(t)){B}_{x}(t)^{-1} :  T_{\gamma_{x}(t)}\Sigma_t\rightarrow T_{\gamma_{x}(t)}\Sigma_t, \\
\mathcal{U}_{x}^{\perp}(t)&:=[\mathcal{U}_{x}(t)]^{\perp}:  T_{\gamma_{x}(t)}\Sigma_t\rightarrow (T_{\gamma_{x}(t)}\Sigma_t)^{\perp},
\end{align*}
where $\nabla_{t}$ denotes the covariant derivative along $\gamma_{x}(t)$ in $M$ and $D_{t}:=\top \circ \nabla_{t}$, $\top$ being the orthogonal projection on $T_{\gamma_{x}(t)} \Sigma_{t}$ and $\perp$ being the orthogonal projection on the orthogonal complement $(T_{\gamma(t)}\Sigma_t)^{\perp}$ of $T_{\gamma(t)}\Sigma_t$.
For every $x\in \Sigma_{1/2}\setminus N$ such that  $|\dot{\gamma}_{x}|\neq 0$,  we define
 \begin{equation}\label{eq:defkappa}
 \kappa_{\gamma_{x}}:[0,|\dot{\gamma}_{x}|]\rightarrow \mathbb{R}, \quad  \kappa_{\gamma_{x}}(|\dot{\gamma_{x}}|\,t) \, | \dot{\gamma}_{x}|^2:= \left\|\mathcal{U}_{x}^{\perp}(t) \right\|^{2}, \quad \forall t \in [0,1],
 \end{equation}
 if $|\dot{\gamma}_{x}|=0$,  we set   $\kappa_{\gamma_{x}}(0)=0$. 
Observe that the map $[0,1]\ni t \mapsto   \kappa_{\gamma_{x}}(|\dot{\gamma_{x}}|\,t) \in \R$ is invariant under constant speed reparametrization of the geodesic $\gamma_{x}$. 

We now introduce the generalized distortion coefficients $\sigma_{\kappa}$ associated to a continuous function $\kappa:[0,\theta] \to \mathbb{R}$ (cf. \cite{ketterer4}). First of all, the generalized $\sin$-function associated to $\kappa$, denoted by $\sin_{\kappa}$,  is defined as the unique solution $v:[0,\theta] \to \R$ of the equation
\begin{align*}
v''+\kappa v=0 \ \ \& \ \ v(0)=0, \ v'(0)=1.
\end{align*}
The generalized distortion coefficients $\sigma_{\kappa}^{(t)}(\theta)$, for $t \in [0,1]$ and $\theta>0$, are defined as 
\begin{align}\label{eq:defGensigma}
\sigma_{\kappa}^{(t)}(\theta):=\begin{cases}
                              \frac{\sin_{\kappa}(t\theta)}{\sin_{\kappa}(\theta)} \ \ & \ \ \text{if} \ \sin_{\kappa}(s \theta)>0 \; { \text{ for all } s \in [0,1], }\\
                            {  \infty } \ \ & \ \ \mbox{otherwise}.
                              \end{cases}
\end{align}
Using Sturm-Picone comparison Theorem one can check that (see for instance \cite [Proposition 3.4]{ketterer4})  
\begin{align}\label{eq:Monotsigma}
\kappa_{1}\leq \kappa_{2}  \text{ on } [0,\theta]  \quad \Longrightarrow \quad    \sigma_{\kappa_{1}}^{(t)}(\theta)\leq  \sigma_{\kappa_{2}}^{(t)}(\theta) \quad \forall  t \in [0,1].
\end{align}
Moreover, by the strong maximum principle (see for instance \cite [XVIII]{Walter}), it holds
\begin{align}\label{eq:Strictmonotsigma}
\kappa_{1} <  \kappa_{2}  \text{ on } (0,\theta) \; \; \& \; \;  \sigma_{\kappa_{1}}^{(\cdot)}(\theta) \not \equiv \infty  \;  \quad \Longrightarrow \quad    \sigma_{\kappa_{1}}^{(t)}(\theta)< \sigma_{\kappa_{2}}^{(t)}(\theta),  \quad \forall  t \in (0,1).
\end{align}
\noindent
It is convenient to also set $\sigma_{\kappa}^{(\cdot)}(0)\equiv 1$,
$\kappa^-(t):=\kappa(\theta-t)$ and $\kappa^{+}(t):=\kappa(t)$.
\\If $v_0,v_1\in [0,\infty)$, a straightforward computation gives that  $v(t):=\sigma^{(1-t)}_{\kappa^-}(\theta)v_0 + \sigma_{\kappa^+}^{(t)}(\theta)v_1$ solves 
\begin{align}\label{equ:SL}
v''(t)+\kappa(t\theta)\theta^2v=0,\, \forall t\in (0,1)\mbox{ with }v(0)=v_0 \ \& \ v(1)=v_1,
\end{align}
provided $t\in [0,1]\mapsto \sigma_{\kappa^+}^{(t)}(\theta)$ (or, equivalently, $t\in [0,1]\mapsto \sigma_{\kappa^-}^{(t)}(\theta)$) is real-valued. 
\\By \cite[Proposition 3.8]{ketterer4}, if $u:[0,1]\to (0,\infty)$ with $u \in C^{0}([0,1])\cap C^{2}((0,1))$   satisfies
\begin{align}\label{eq:SubSol}
u''(t)+\kappa(t\theta)\theta^2 \, u(t)\leq 0,\, \forall t\in (0,1)\mbox{ with }u(0)=v_0 \ \& \ u(1)=v_1 \quad \Longrightarrow \quad u\geq v \text{ on } [0,1].
\end{align}
It also convenient to consider a slightly different comparison function. To this aim we define the function ${\rm g}:[0,1]\times [0,1]\to [0,1]$ by
\begin{equation}
  \label{eq:defgreen}
  {\rm g}(s,t):=
  \begin{cases}
    (1-s)t&\text{if }t\in[0,s],\\
    s(1-t)&\text{if }t\in [s,1],
  \end{cases}
\end{equation}
so that for all $s\in (0,1)$ one has
\begin{equation}
  \label{eq:greenindentity}
  -\frac{\partial^2}{\partial t^2}  {\rm g}(s,t) =\delta_{s}\quad\text{in $\mathscr D'(0,1)$},\qquad 
   {\rm g}(s,0)= {\rm g}(s,1)=0.
\end{equation}
Given $w_0,w_1\in [0,\infty)$ and a continuous function $u:[0,1] \to [0,\infty)$, a straightforward computation gives that  $w(t):=(1-t) w_{0} + t w_{1} +\int_{0}^{1} {\rm g}(s,t)\,   u(s) \, d s$ solves 
\begin{align}\label{equ:SLw}
w''(t)+ u(t)=0,\, \forall t\in (0,1)\mbox{ with }w(0)=w_0 \ \& \ w(1)=w_1.
\end{align}

\begin{theorem}[OT Characterization of curvature lower bounds]\label{th:lower}
Let $(M,g)$ be a complete Riemannian manifold with $\partial M=\emptyset$ and let $K\in\mathbb{R}$. Then the following statements are equivalent: 
\begin{itemize}
\item[(i)] $\ric_{p}\geq K$.
\medskip
\item[(ii)]
Let $\{\mu_t\}_{t\in [0,1]}$ be a countably $\mathcal{H}^p$-rectifiable $W_2$-geodesic, and let $\Pi$ be the corresponding dynamical optimal plan. Then, for any $p'\geq p$, it holds
\begin{align}\label{eq:Sp'mutHpThm}
S_{p'}(\mu_t|\mathcal{H}^p)\leq - \int \left[\sigma_{(K-\kappa_{\gamma}^{-})/p'}^{(1-t)}(|\dot{\gamma}|)\, \rho_0^{-\frac{1}{p'}}(\gamma(0)) + \sigma_{(K-\kappa_{\gamma}^+)/p'}^{(t)}(|\dot{\gamma}|) \, \rho_1^{-\frac{1}{p'}}(\gamma(1))\right]d\Pi(\gamma), \;
\forall t\in[0,1]
\end{align}
where $\kappa_{\gamma}$ was defined in \eqref{eq:defkappa} and the generalized distortion coefficients $\sigma$ are as in \eqref{eq:defGensigma}.
\item[(iii)]
Let $\{\mu_t\}_{t\in [0,1]}$ and  $\Pi$ be as in (ii). Then
\begin{align*}
\Ent(\mu_t|\cH^{p})\leq (1-t)\Ent(\mu_0|\cH^{p})+t\Ent(\mu_1|\cH^{p}) - \int \int_0^1 {\rm g}(s,t)\, |\dot{\gamma}|^2 \,  (K-\kappa_{\gamma}(s|\dot{\gamma}|) )\,  ds \, d\Pi(\gamma), \; \forall t\in[0,1]
\end{align*}
where ${\rm g}(s,t)$ was defined in \eqref{eq:defgreen}.
\end{itemize}
\end{theorem}

\begin{remark}\label{rem:thmSharp}
We emphasize that Theorem \ref{th:lower} is sharp. First of all, one can not omit the correction term $\kappa_{\gamma}$: even in $\mathbb{R}^n$, the convexity of $S_{p}$ is not true in general. 
For instance consider $\mathbb{R}^2$ and the line segment $\left\{(t,\frac{1}{2}t):t\in [0,1]\right\}=:L_0$ and let $\mu_0=\mathcal{H}^1|_{L_0}$; similarly, define $\mu_1=\mathcal{H}^1|_{L_1}$ where $L_1:=\left\{(t,-\frac{1}{2}t):t\in [0,1]\right\}$. 
Then, it is easy to check that the 
optimal transport between $\mu_0$ and $\mu_1$ is supported on geodesics that connect $(t,\frac{1}{2}t)$ and $(t,-\frac{1}{2}t)$ and $\mu_{1/2}$ is exactly $\mathcal{H}^{1}|_{[0,1]\times\left\{0\right\}}$.
If Theorem \ref{th:lower} would hold with $K=0$ and $\kappa_{\gamma} \equiv 0$,  then the Brunn-Minkowski inequality (see Corollary \ref{cor:BM} below) would contradict that  the $\mathcal{H}^1$-measure of $[0,1]\times \left\{0\right\}$ is strictly smaller than the one of $L_0$ and $L_1$.
\\Second, we stress that the arguments in the proof of  Theorem \ref{th:lower}  are sharp, since for this example all the inequalities become identities (for the details see Remark \ref{rem:thmisharp}  after the proof).
\end{remark}

The proof of Theorem \ref{th:lower} will make use of the next proposition.
\begin{proposition}\label{C}
Let $M$ be a complete Riemannian $n$-dimensional manifold without boundary. Assume that   $\ric_{p}\geq K$, for some $p\in \{1,\ldots,n\}$ and $K\in \R$, and consider a countably $\mathcal{H}^p$-rectifiable $W_{2}$-geodesic $\{\mu_t\}_{t\in [0,1]}$.  
\\Then, using the notation recalled at the beginning of Section \ref{sec:LB}  and denoting  $\mathcal{J}_x(t):=\det [B_x(t)]$, it holds
\begin{equation}%
\frac{d^2}{dt^2}\mathcal{J}^{\frac{1}{p'}}_x(t)\leq - \frac{\left(K-\kappa_{\gamma_x(t)}\right)}{p'}|\dot{\gamma}_{x}|^2 \mathcal{J}^{\frac{1}{p'}}_x(t),  \; \forall x\in \Sigma_{1/2}\setminus N, \, \mu_{1/2}(N)=0, \; \forall p'\geq p, \; \forall t \in (0,1),
\end{equation}
and thus
\begin{align}
\cJ^{\frac{1}{p'}}_x(t)& \geq \sigma_{\frac{K-\kappa^-_{\gamma_x}}{p'}}^{(1-t)}(|\dot{\gamma}_{x}|) \, \cJ^{\frac{1}{p'}}_x(0)+\sigma_{\frac{K-\kappa_{\gamma_x}^+}{p'}}^{(t)}(|\dot{\gamma}_{x}|) \, \cJ^{\frac{1}{p'}}_x(1) . \quad   \forall x \in \Sigma_{1/2}\setminus N, \; \forall t \in [0,1], \;  \forall p'\geq p.  \label{eq:Jp'sigmamain} 
\end{align}
\end{proposition}
\begin{proof}
If we set $y_x(t)=\log \cJ_x(t)=\log\mbox{det}B_x(t)$, from \eqref{eq:eqdiffyx} in Proposition \ref{cor:imp} we know that 
\begin{equation*}
y_{x}''(t)+ \frac{1}{p}y_{x}'(t)^2  +  \ric_p(T_{\gamma_{x}(t)}\Sigma_t,\dot{\gamma_{x}}(t)) - \| \mathcal{U}^{\perp}(t)  \|^{2} \leq 0, \quad   \forall t \in (0,1), \quad  \forall x \in \Sigma_{1/2}\setminus N.
\end{equation*}
Plugging the assumption  $ \ric_p\geq K$ together with the definition \eqref{eq:defkappa} of  $\kappa_{\gamma_{x}}$, we get
\begin{align}\label{eq:yx''proof}
y_x''(t)+ \frac{1}{p'} y_x'(t)^2 + (K-\kappa_{\gamma_x}(t))|\dot{\gamma}_x(t)|^2\leq 0,  \quad  \forall t \in (0,1), \quad  \forall x \in \Sigma_{1/2}\setminus N, \quad  \forall p'\geq p,
\end{align}
which is equivalent to
\begin{align}\label{eq:j1p'proof}
(\cJ^{\frac{1}{p'}})''_x(t)+\frac{K-\kappa_{\gamma_x}(t)}{p'}|\dot{\gamma}_x(t)|^2\cJ^{\frac{1}{p'}}(t)\leq 0 , \quad  \forall t \in (0,1), \quad  \forall x \in \Sigma_{1/2}\setminus N, \quad  \forall p'\geq p.
\end{align}
The claimed  \eqref{eq:Jp'sigmamain} follows then by the comparison principle \eqref{eq:SubSol} and by Proposition 3.8 in \cite{ketterer4}.
More precisely,  in  step \textbf{3} of the 
proof of \cite[Proposition 3.8]{ketterer4} it is showed that  if $\cJ^{\frac{1}{p'}}$ satisfies \eqref{eq:j1p'proof} and  
$\cJ^{\frac{1}{p'}}(t)>0$ for some $t\in [0,1]$ then $\sigma_{(K-\kappa_{\gamma}^{\pm})/p}^{(t)}(|\dot{\gamma}|)<\infty$; the desired \eqref{eq:Jp'sigmamain} follows then from \eqref{eq:SubSol}.  
\end{proof}

\noindent
{\it Proof of Theorem \ref{th:lower}.} (i)$\implies$(ii). Let $\{\mu_t\}_{t\in [0,1]}$ be a countably $\mathcal{H}^p$-rectifiable geodesic.
Recall that for every $t \in [0,1]$ it holds $\mu_t=\rho_{t} \, {\mathcal H}^{p}\llcorner \Sigma_{t}=(T^{t}_{1/2})_{\sharp} \mu_{1/2}$. Let $\Pi$ be the optimal dynamical plan associated to the 
$W_{2}$-geodesic $\{\mu_t\}_{t\in [0,1]}$, i.e. $\mu_{t}=(\ee_{t})_{\sharp} \Pi$.
\\Setting  $\cJ_x(t)=\det B_x(t)= \det [DT_{1/2}^{t} (x)]$, for all $t\in (0,1)$ 
and $p'\geq p$ we get:
\begin{align*}
&\int_{\Sigma_{t}}\rho_t(y)^{-\frac{1}{p'}}d\mu_t(y)= \int_{\Sigma_{1/2}} \rho_t(T_{1/2}^{t}(x))^{-\frac{1}{p'}} \, d\mu_{1/2}(x) \\
&\overset{\eqref{eq:MAineq}}{=} \int_{\Sigma_{1/2}} \cJ^{\frac{1}{p'}}_x(t)\, \rho_{1/2}(x)^{1-\frac{1}{p'}} \,d\mathcal{H}^p(x)\\
&\overset{\eqref{eq:Jp'sigmamain}}{\geq} \int_{\Sigma_{1/2}}\Big[\sigma_{(K-\kappa_{\gamma_x})^-/p'}^{(1-t)}(|\dot{\gamma}_x|) \, \cJ^{\frac{1}{p'}}_x(0) \, \rho_{1/2}(x)^{1-\frac{1}{p'}}+\sigma_{(K-\kappa_{\gamma_x}^+)/p'}^{(t)}(|\dot{\gamma}_x|) \, \cJ^{\frac{1}{p'}}_x(1) \, \rho_{1/2}(x)^{1-\frac{1}{p'}}\Big]d\mathcal{H}^p(x)\\
&\overset{\eqref{eq:MAineq}}{\geq} \int_{\Sigma_{1/2}}\Big[\sigma_{(K-\kappa_{\gamma_x})^-/p'}^{(1-t)}(|\dot{\gamma}_x|)\rho_0(T_{1/2}^{0}(x))^{-\frac{1}{p'}}+\sigma_{(K-\kappa_{\gamma_x})^+/p'}^{(t)}(|\dot{\gamma}_x|)\rho_1(T_{1/2}^{1}(x))^{-\frac{1}{p'}}\Big]\rho_{1/2}(x) \, d\mathcal{H}^p(x)\\
&\overset{(\ref{eq:trans})}{=} \int\left[\sigma_{(K-\kappa_{\gamma})^-/p'}^{(1-t)}(|\dot{\gamma}|)\rho_0(\gamma(0))^{-\frac{1}{p'}}+\sigma_{(K-\kappa_{\gamma})^+/p'}^{(t)}(|\dot{\gamma}|)\rho_1(\gamma(1))^{-\frac{1}{p'}}\right]d\Pi(\gamma).
\end{align*}
This concludes the proof of \eqref{eq:Sp'mutHpThm} for  $t\in (0,1)$. In case $t=0$ or $t=1$ just observe that  from the  very definition \eqref{eq:defGensigma} it holds $\sigma_{(K-\kappa_{\gamma})^-/p'}^{(0)}(|\dot{\gamma}|)=0$ and $\sigma_{(K-\kappa_{\gamma})^-/p'}^{(1)}(|\dot{\gamma}|)=1$, so the claim \eqref{eq:Sp'mutHpThm} is trivially satisfied.
\smallskip

(ii)$\implies$(i).
\\ We argue by contradiction. Assume there exist $x_{0}\in M$, 
a $p$-dimensional plane $P\subset T_{x_{0}}M$ and $0\neq v\in T_{x_{0}}M$ such that 
the $p$-Ricci curvature of $P$ in the direction of $v$ satisfies 
\begin{equation}
\ric_p(P,v)\leq(K-4\epsilon)|v|^2, 
\end{equation}
 for some $\epsilon>0$.
Let $\delta>0$ be sufficiently small such that $\exp_{x_{0}}|_{B_{\delta}(0)}$ is a diffeomorphism onto its image. Then
$\exp_{x_{0}}(P\cap B_{\delta}(0))=:\Sigma_{\scriptscriptstyle{\frac{1}{2}}}$ is a smooth $p$-dimensional submanifold. 
Let
$\phi\in C^{\infty}_0(M)$ be a Kantorovich potential such that 
\begin{equation}\label{eq:AssPhiNonDeg}
\nabla \phi(x_{0}) = v \neq 0 \quad \text{ and }  \quad \nabla^2\phi(x_{0})=0. 
\end{equation}
By replacing $\phi$ with $\eta\phi$ for a sufficiently small number $\eta>0$ we get that $\phi$ is a Kantorovich potential as well and $|\nabla \phi|(y)$ 
is  smaller than the injectivity radius at $y$, for every $y\in \supp(\phi)\subset M$.  It is easily checked that for $\delta>0$ small enough the map
$y\mapsto T_t(y)=\exp_y(-t\nabla \phi(y))$ is 
a diffeomorphism from $B_{\delta}(0)$ onto its image for any $t\in [-\frac{1}{2},\frac{1}{2}]$. Hence, $\Sigma_t:=T_{t-\frac{1}{2}}(\Sigma_{\frac{1}{2}})$ for $t\in [0,1]$ is a 1-parameter family 
of smooth $p$-dimensional submanifolds with finite $p$-dimensional Hausdorff measure. We define $\mu_{\frac 1 2 }:= \mathcal{H}^p(\Sigma_{\frac 1 2 })^{-1} \,\mathcal{H}^p \llcorner \Sigma_{\frac 1 2}$;
note that  $\mu_{t}:=(T_{t-\frac{1}{2}})_{\sharp}\mu_{\frac  1 2}$, with $t \in [0,1]$, is the unique $L^2$-Wasserstein geodesic between $\mu_0$ and $\mu_1$. 
Moreover, by construction, $\mu_t$ is a $\mathcal{H}^p$-absolutely continuous probability measure concentrated on $\Sigma_{t}$.
\\Calling $\gamma_{x}(t):=T_{t-\frac{1}{2}}(x)=\exp_{x}\big(-\big(t-\frac{1}{2}\big) \nabla \phi(x)\big)$ for $x \in \Sigma_{\frac 1 2}$ the geodesic performing the transport,  note  that by continuity  there exist  $\delta, \sigma>0$ small enough such that  
\begin{equation}\label{eq:AssAbsRicp}
\ric_p(T_{\gamma_{x}(t)}\Sigma_{t}, \dot{\gamma}_{x}(t) )<  (K-3 \epsilon) |\dot{\gamma}_{x}(t)|^{2}, \quad \forall x \in \Sigma_{\frac 1 2}\subset B_{\delta}(x_{0}), \quad \forall  t \in \left[\frac{1}{2}-\sigma, \frac{1}{2}+\sigma \right] . 
\end{equation}
The identity \eqref{eq:truT'} proved in Proposition \ref{cor:imp} reads as
 \begin{align}\label{ricattiContr}
\tr[\mathcal{U}^{\top}_{x}(t)]'+ \tr[({\mathcal{U}}^{\top}_{x}(t)^2] + \ric_p(T_{\gamma_{x}(t)}\Sigma_t,\dot{\gamma}_{x}(t))= \|\mathcal{U}_{x}^{\perp}(t)\|^{2},  \quad \forall x \in \Sigma_{\frac 1 2}, \quad \forall t \in [0,1]. 
\end{align} 
Since by construction $\mathcal{U}_{x_{0}}(0):=\nabla_{t}B_{x_{0}}(0)B_{x_{0}}^{-1}(0)= \nabla^2\phi (x_{0})|_{T_{x_{0}}\Sigma_{\frac 1 2}}=0$ and $v \neq 0$,  again by continuity we can choose $\delta,\sigma>0$ even smaller so that
\begin{align}\label{eq:claim}
 \|\mathcal{U}_{x}^{\perp}(t)\|^{2}+  \tr[({\mathcal{U}}^{\top}_{x}(t)^2]=\sum_{i=1}^p\left[|\mathcal{U}^{\perp}(t)E_i|^2+|\mathcal{U}^{\top}(t)E_i|^2\right]=\sum_{i=1}^p|\mathcal{U}(t)E_i|^2= \|\mathcal{U}_{x}(t)\|^{2}< \epsilon |\dot{\gamma}_{x}(t)|^{2}
\end{align}
for all $x \in \Sigma_{\frac 1 2} \subset B_{\delta}(x_{0})$ and all $t \in \left[  \frac 1 2 -\sigma , \frac 1 2 +\sigma  \right].$
The combination of  \eqref{eq:AssAbsRicp}, \eqref{ricattiContr} and   \eqref{eq:claim} yields
\begin{align} \label{eq:TrnoU2}
0 \leq  \|\mathcal{U}_{x}^{\perp}(t)\|^{2}< \mbox{tr}[\mathcal{U}^{\top}_x(t)]'+  (K- 2\epsilon) |\dot{\gamma}_{x}(t)|^{2} \leq \mbox{tr}[\mathcal{U}^{\top}_x(t)]'+\frac{1}{p}\mbox{tr}[\mathcal{U}^{\top}_x(t)]^2+\left(K- 2 \epsilon  \right)|\dot{\gamma}_x(t)|^2, 
\end{align}
for all $x \in \Sigma_{\frac 1 2}$ and all $t \in \left[  \frac 1 2 -\sigma , \frac 1 2 +\sigma  \right].$
Observe that the affine reparametrization $t=g(s)=  \frac 1 2 - \sigma + 2\sigma s$, $g:[0,1]\to \left[ \frac 1 2 - \sigma , \frac 1 2+ \sigma  \right]$, corresponds to consider the rescaled Kantorovich potential $2\sigma \phi$ in place of $\phi$ in the arguments above, and  thus gives  
\begin{align*}
\mbox{tr}[\mathcal{U}^{\top}_x(g(s))]'+\frac{1}{p}\mbox{tr}[\mathcal{U}^{\top}_x(g(s))]^2+\left(K- 2 \epsilon  \right)|\dot{\gamma}_x(g(s))|^2 &= 4\sigma^{2} \left[ \mbox{tr}[\mathcal{U}^{\top}_x(t)]'+\frac{1}{p}\mbox{tr}[\mathcal{U}^{\top}_x(t)]^2+\left(K- 2 \epsilon  \right)|\dot{\gamma}_x(t)|^2    \right] \nonumber \\
&> 0 , \qquad  \forall x \in \Sigma_{\frac 1 2}, \; \forall s \in [0,1]. 
\end{align*}

Arguing as in the proof of Proposition \ref{C} (but with reversed inequalities), the last differential inequality gives
\begin{equation}\label{eq:Jp'sigmaRev}
\cJ^{\frac{1}{p}}_x(g(s))\leq \sigma_{\frac{K- 2 \epsilon}{p}}^{(1-t)}(|\dot{\gamma}_{x}\circ g|) \, \cJ^{\frac{1}{p}}_x(0)+\sigma_{\frac{K-2 \epsilon}{p}}^{(t)}(|\dot{\gamma}_{x}\circ g|) \, \cJ^{\frac{1}{p}}_x(1) , \quad   \forall x \in \Sigma_{1/2}\setminus N, \quad \forall s \in [0,1].
\end{equation}
Note in particular that, since $\cJ^{\frac{1}{p}}_x(g(s))>0$ for all $s \in [0,1]$, then $\sigma_{\frac{K- 2 \epsilon}{p}}^{(\cdot)}(|\dot{\gamma}_{x}\circ g|) \not \equiv 0$.

Since $g$ is affine, the restricted and rescaled curve $\{\tilde{\mu}_{s}:=\mu_{g(s)}\}_{s \in [0,1]}$ is still a $W_{2}$-geodesic from
$\tilde{\mu}_{0}=\mu_{\frac 1 2 -\sigma}$ to  $\tilde{\mu}_{1}=\mu_{\frac 1 2 +\sigma}$.
By repeating the arguments  in the proof of (i)$\implies$(ii), with reversed inequalities (note that \eqref{eq:MAineq} holds with equality since we are considering the interior of a geodesic,  
use \eqref{eq:Jp'sigmaRev} instead of  \eqref{eq:Jp'sigmamain}, and replace  $K-\kappa_{\gamma}$  by $K-2\epsilon $),  we obtain
\begin{equation}\label{ineq}
\int_{\Sigma_{\frac 1 2}} \tilde{\rho}_{\frac 1 2}(y)^{-\frac {1} {p}} d \tilde{\mu}_{\frac 1 2}(y) \leq \int \Big[ \sigma_{(K- 2\epsilon)/p}^{\left(\frac{1}{2}\right)}(|\dot{\gamma}|) \tilde{\rho}_0(\gamma(0))^{-\frac{1}{p}} +\sigma_{(K-2 {\epsilon})/p}^{\left(\frac{1}{2}\right)} (|\dot{\gamma}|) \tilde{\rho}_1(\gamma(1))^{-\frac{1}{p}} \Big] d\tilde{\Pi}(\gamma),
\end{equation}
where $\tilde{\Pi}$ is the dynamical  optimal associated to $\{\tilde{\mu}_{s}=\tilde{\rho}_{s} \cH^{p}\}_{s \in [0,1]}$.
\\
Observing that  \eqref{eq:claim} gives $\kappa_{\gamma}\leq\epsilon$ and noting that $\left\|\mathcal{U}^{\perp}_x(t)\right\|^2$ 
correctly scales when we apply the the reparametrization $g$),  using \eqref{eq:Monotsigma}  we get that \eqref{ineq} implies
\begin{align*}
\int_{\Sigma_{\frac 1 2}} \tilde{\rho}_{\frac 1 2}(y)^{-\frac {1} {p}} d \tilde{\mu}_{\frac 1 2}(y)& \leq \int \Big[ \sigma_{(K-\kappa_{\gamma_{x}}- \epsilon)/p}^{\left(\frac{1}{2}\right)}(|\dot{\gamma}|) \tilde{\rho}_0(\gamma(0))^{-\frac{1}{p}} +\sigma_{(K-\kappa_{\gamma_{x}}- {\epsilon})/p}^{\left(\frac{1}{2}\right)} (|\dot{\gamma}|) \tilde{\rho}_1(\gamma(1))^{-\frac{1}{p}} \Big] d\tilde{\Pi}(\gamma)    \nonumber\\
&\overset{\eqref{eq:Strictmonotsigma}} {<}  \int \Big[ \sigma_{(K-\kappa_{\gamma_{x}})/p}^{\left(\frac{1}{2}\right)}(|\dot{\gamma}|) \tilde{\rho}_0(\gamma(0))^{-\frac{1}{p}} +\sigma_{(K-\kappa_{\gamma_{x}})/p}^{\left(\frac{1}{2}\right)} (|\dot{\gamma}|) \tilde{\rho}_1(\gamma(1))^{-\frac{1}{p}} \Big] d\tilde{\Pi}(\gamma).
\end{align*}
This contradicts (ii) for the geodesic $\{\tilde{\mu}_{s}=\tilde{\rho}_{s} \cH^{p}\}_{s \in [0,1]}$.
\\

(i) $\Longrightarrow$ (iii).
\\For $t \in (0,1)$ we have
\begin{align}
\Ent(\mu_t|\cH^p) &=\int_{\Sigma_{t}}  \log\rho_{t}(y) \, d\mu_{t}(y) = \int_{\Sigma_{1/2}}  \log\rho_{t}(T_{1/2}^{t}(x)) \, d\mu_{1/2}(x)  \overset{ \eqref{eq:MAineq}}{=} \int_{\Sigma_{1/2}}  \log[\rho_{1/2}(x) \cJ_{t}(x)^{-1}]\, d\mu_{1/2}(x) \nonumber \\
&= \Ent(\mu_{1/2}|\cH^{p})-   \int_{\Sigma_{1/2}}  y_{x}(t)\, d\mu_{1/2}(x),  \label{eq:EntHpy}
\end{align}
where  $y_{x}(t)=\log( \cJ_{t}(x))$. Using  \eqref{eq:yx''proof} we obtain
\begin{equation}
\frac{d^{2}}{dt^{2}}\Ent(\mu_t|\cH^p)\geq \int_{\Sigma_{1/2}} \big( K- \kappa_{\gamma_{x}} (t) \big)  |\dot{\gamma}_{x}(t)|^{2}\, d\mu_{1/2}(x), \quad \forall t\in (0,1).
\end{equation}
We then get (iii) using \eqref{equ:SLw} and standard comparison.
\\

(iii) $\Longrightarrow$ (i). 
\\Assume that by contradiction \eqref{eq:AssAbsRicp} holds and repeat verbatim the first part of the proof of (ii) $\Longrightarrow$ (i) to reach \eqref{eq:TrnoU2}, i.e.
\begin{align*} \label{eq:TrnoU2inf}
\mbox{tr}[\mathcal{U}^{\top}_x(t)]'+  (K- 2\epsilon) |\dot{\gamma}_{x}(t)|^{2} >0, \quad \forall x \in \Sigma_{\frac 1 2}, \; \forall t \in \left[  \frac 1 2 -\sigma , \frac 1 2 +\sigma  \right].
\end{align*}
Considering as above  the affine reparametrization $t=g(s)=  \frac 1 2 - \sigma + 2\sigma s$, $g:[0,1]\to \left[ \frac 1 2 - \sigma , \frac 1 2+ \sigma  \right]$ and recalling \eqref{eq:y'}, we obtain
\begin{align*}
y_{x}(g(s))''+\left(K- 2 \epsilon  \right)|\dot{\gamma}_x(g(s))|^2> 0 , \qquad  \forall x \in \Sigma_{\frac 1 2}, \; \forall s \in [0,1]. 
\end{align*}
Calling as above $\{\tilde{\mu}_{s}:=\mu_{g(s)}\}_{s \in [0,1]}$ the corresponding rescaled  $W_{2}$-geodesic,  the combination of the last inequality with \eqref{eq:EntHpy} gives
\begin{equation}\label{eq:Entmugs''}
\frac{d^{2}}{ds^{2}}\Ent(\tilde{\mu}_s|\cH^{p})< \int_{\Sigma_{1/2}} \big( K- 2 \epsilon \big)  |\dot{\gamma}_{x}(g(s))|^{2}\, d\mu_{1/2}(x), \quad \forall s\in (0,1).
\end{equation}
Calling $\tilde{\Pi}$ the dynamical optimal plan associated to the geodesic  $\{\tilde{\mu}_{s}:=\mu_{g(s)}\}_{s \in [0,1]}$, using \eqref{equ:SLw} and standard comparison we get that
\begin{align*}
\Ent(\tilde{\mu}_s|\cH^{p})> (1-s)\, \Ent(\tilde{\mu}_0|\cH^{p})+s \, \Ent(\tilde{\mu}_1|\cH^{p}) - \int \int_0^1 {\rm g}(t,s)\, |\dot{\gamma}|^2 \,  (K- 2 \epsilon )\,  dt \, d\tilde{\Pi}(\gamma).
\end{align*}
Observing now that  \eqref{eq:claim} gives $\kappa_{\gamma}\leq\epsilon$ for $\tilde{\Pi}$-a.e. $\gamma$, we obtain
\begin{align*}
\Ent(\tilde{\mu}_s|\cH^{p})> (1-s)\, \Ent(\tilde{\mu}_0|\cH^{p})+s \, \Ent(\tilde{\mu}_1|\cH^{p}) - \int \int_0^1 {\rm g}(t,s)\, |\dot{\gamma}|^2 \,  (K- \kappa_{\gamma}(t |\dot{\gamma}|)-  \epsilon )\,  dt \, d\tilde{\Pi}(\gamma),
\end{align*}
which contradicts (iii) thanks to the strict positivity of  ${\rm g}$  on $(0,1)\times (0,1)$.
\qed
\begin{remark}\label{rem:thmisharp}
In order to show that Theorem \ref{th:lower} is sharp, we show that equality is achieved in \eqref{eq:Sp'mutHpThm} for the example of Remark \ref{rem:thmSharp},  $p'=1$.
First of all recall that, in euclidean spaces, the Jacobi fields are affine functions along the geodesics. The initial measure $\mu_0$ is supported on the
segment $\left\{(t,\frac{1}{2}t):t\in [0,1]\right\}$ that is generated by the unit vector $\frac{1}{\sqrt{5}}(2,1)=e$, and the final measure is supported on the segment $\left\{(t,-\frac{1}{2}t):t\in[0,1]\right\}$ that is 
generated by $-\frac{1}{\sqrt{5}}(2,1)$. Set $(\frac{2}{\sqrt{5}},0)=v$ and $(0,\frac{1}{\sqrt{5}})=w$. 
We have $$B_x(t)=J_e(t)= v+ (1-2t)w =\left(\frac{2}{\sqrt{5}},(1-2t)\frac{1}{\sqrt{5}}\right),$$
and $J'_e(t)=-2w=\left(0,-\frac{2}{\sqrt{5}}\right)$.
Clearly, $u(t)=(t-\frac{1}{2},1)$ is orthogonal to $J_e(t)$ for every $t\in [0,1]$. Thus, $$\left\|(J'_e(t))^{\perp}\right\|=\frac{1}{\left\|u(t)\right\|}\langle J'_e(t), u(t)\rangle=- \frac{2}{\sqrt{5(t^2-t+\frac{5}{4})}}.$$
Using the identity \eqref{eq:UtperpJe}, we get that
$$\kappa_{\gamma}(t|\dot{\gamma}_x|)|\dot{\gamma}_x|^2=\left\|(\mathcal{U}(t)E(t))^{\perp}\right\|^2=\left\|\left\|J_e(t)\right\|^{-1}(J'_e(t))^{\perp}\right\|^2=\frac{1}{\left(t^2-t+\frac{5}{4}\right)^{2}}=\kappa(t)$$
where $E(t)=\left\|J_e(t)\right\|^{-1}J_e(t)$. It follows that the  coefficient $\kappa_{\gamma}(t|\dot{\gamma}|)|\dot{\gamma}|^2$ does not depend on $\gamma$.  We thus get
\begin{align}\label{eq:H1ex}
\mathcal{H}^1(\supp \mu_t)=\sigma_{-\kappa^-}^{(1-t)}(1) \, \mathcal{H}^1(\supp \mu_0)+\sigma_{-\kappa^+}^{(t)}(1)\, \mathcal{H}^1(\supp \mu_1).
\end{align}
Indeed,  the $\mathcal{H}^1$-measure of the support of $\mu_{t}$ is given by the length of the Jacobi field $J_e(t)$ with the normalisation $\alpha>0$ 
such that $\alpha J_e(\frac{1}{2})=(1,0)$:
\begin{align*}
 \mathcal{H}^1(\supp \mu_t)=\sqrt{t^{2} -t+ \frac{5}{4}}.
\end{align*}
A straightforward computation shows that $t\mapsto  \mathcal{H}^1(\supp \mu_t)$  solves the boundary value problem $f''(t)=\kappa(t)f(t)$, $f(1)=f(0)=\mathcal{H}^1(\supp \mu_0)=\mathcal{H}^1(\supp \mu_1)=\frac{\sqrt{5}}{2}$, thus \eqref{eq:H1ex} follows.
\\Note that, for this example,  in the proof of Theorem \ref{th:lower} every inequality becomes an identity, showing the sharpness of the arguments.
\end{remark}

As an application of Theorem \ref{th:lower} we establish  a  new Brunn-Minkowski type inequality involving the $\cH^{p}$-measure and countably $\cH^{p}$-rectifiable sets. The main novelty is about the measure: the standard Brunn-Minkowski inequality involves the top Hausdorff measure $\cH^{n}$ in an $n$-dimensional Riemannian manifold. A second refinement is that, 
in comparison  with the standard Brunn-Minkowski inequality where one gives a lower bound on the measure of \emph{all intermediate points}, here we  give more sharply  a lower bound on the measure of just the  \emph{$t$-intermediate points where the optimal transport is performed} (let us mention that this was already the case in  \cite{stugeo2} even if not explicitly stated, but there one considers the top dimensional Hausdorff measure).
  
\begin{corollary}[$p$-Brunn-Minkowski inequality]\label{cor:BM}
Let $(M,g)$ be a complete $n$-dimensional Riemannian manifold without boundary. Assume that  $\ric_{p}\geq K$ for some $p\in \{1,\ldots,n\}$ and $K \in \R$.
Let $A_{0},A_{1}\subset M$ be bounded  $p$-rectifiable subsets with positive and finite $\cH^{p}$-measure. Set $\mu_i=\mathcal{H}^p(A_i)^{-1}\mathcal{H}^p \llcorner A_i$ for $i=0,1$ and assume that there exists a $W_{2}$-geodesic $\{\mu_{t}\}_{t \in [0,1]}$ such that for some $t_{0}\in (0,1)$ the measure $\mu_{t_{0}}$ is concentrated on a countably $\cH^{p}$-rectifiable subset $\Sigma_{t_{0}}\subset M$.

Then for every $t \in [0,1]$ one has $\mu_{t}=\rho_{t} \cH^{p} \in  \mathcal{P}_{c}(M,\cH^{p})$ and it holds
\begin{equation}\label{eq:BMHpPi}
\mathcal{H}^{p} \left( \{ \rho_{t} >0 \}\right)^{\frac{1}{p'}}\geq  \int \left[{\sigma_{(K-\kappa_{\gamma}^{-})/p'}^{(1-t)}(|\dot{\gamma}|)}\, \rho_0^{-\frac{1}{p'}}(\gamma(0)) + {\sigma_{(K-\kappa_{\gamma}^+)/p'}^{(t)}(|\dot{\gamma}|)  } \, \rho_1^{-\frac{1}{p'}}(\gamma(1))\right]d\Pi(\gamma), \quad \forall p'\geq p.
\end{equation}
In particular,  calling  
$$A_t:=\{ \gamma(t) \,:\, \gamma \in \Geo(X), \, \gamma(0)\in A_{0}, \, \gamma(1)\in A_{1}\}$$
 the set of all $t$-intermediate points of geodesics with endpoints in $A_0$ and $A_1$, it holds
\begin{equation}\label{eq:BMHp}
\mathcal{H}^{p}(A_t)^{\frac{1}{p'}}\geq  \int \left[{\sigma_{(K-\kappa_{\gamma}^{-})/p'}^{(1-t)}(|\dot{\gamma}|)}\, \rho_0^{-\frac{1}{p'}}(\gamma(0)) + {\sigma_{(K-\kappa_{\gamma}^+)/p'}^{(t)}(|\dot{\gamma}|)  } \, \rho_1^{-\frac{1}{p'}}(\gamma(1))\right]d\Pi(\gamma), \quad \forall p'\geq p.
\end{equation}
\end{corollary}
\begin{proof} From Lemma \ref{rem:mutPreiss} we know that  for every $t \in [0,1]$ it holds $\mu_{t}=\rho_{t} \cH^{p} \in  \mathcal{P}_{c}(M,\cH^{p})$. Moreover  $\cH^{p}  \left( \{ \rho_{t} >0 \} \setminus A_{t}\right)=0$. Therefore, if we prove \eqref{eq:BMHpPi} then also \eqref{eq:BMHp} will follow.
\\In order to get  \eqref{eq:BMHpPi}, observe that from Theorem  \ref{th:lower}   the $W_{2}$-geodesic $\{\mu_{t}\}_{t \in [0,1]}$ satisfies 
\begin{align}
\int_{\{\rho_{t}>0\}}\rho_t(x)^{-\frac{1}{p'}}d\mu_t(x) \geq  \int \left[{\sigma_{(K-\kappa_{\gamma}^{-})/p'}^{(1-t)}(|\dot{\gamma}|)}\, \rho_0^{-\frac{1}{p'}}(\gamma(0)) + {\sigma_{(K-\kappa_{\gamma}^+)/p'}^{(t)}(|\dot{\gamma}|)  } \, \rho_1^{-\frac{1}{p'}}(\gamma(1))\right]d\Pi(\gamma), \quad \forall p'\geq p, \label{eq:pfBM1}
\end{align}
On the other hand, from Jensen inequality we get
\begin{align}
\int_{\{\rho_{t}>0\}}\rho_t(x)^{-\frac{1}{p'}}d\mu_t(x)&= \cH^{p}(\{\rho_{t}>0\}) \;  \int_{\{\rho_{t}>0\}} \rho_{t}^{1- \frac{1} {p'}}  \, d\left(  \frac{1}{\cH^{p}(\{\rho_{t}>0\})}  \cH^{p} \llcorner \{\rho_{t}>0\} \right)   \nonumber \\
&\leq   \cH^{p}(\{\rho_{t}>0\}) \; \left(   \int_{\{\rho_{t}>0\}} \rho_{t}  \, d\left(  \frac{1}{\cH^{p}(\{\rho_{t}>0\})}  \cH^{p} \llcorner \{\rho_{t}>0\} \right) \right)^{1- \frac{1} {p'}}  \nonumber \\
&=   \cH^{p}(\{\rho_{t}>0\})^{\frac{1} {p'}}. \label{eq:pfBM2}
\end{align}
The combination of  \eqref{eq:pfBM1} and \eqref{eq:pfBM2} implies \eqref{eq:BMHpPi}.
\end{proof}

\end{document}